\documentclass[12pt]{article}
\usepackage{graphicx}
\usepackage[utf8]{inputenc}
\usepackage{amsfonts,amssymb}
\usepackage{graphics}
\usepackage{graphicx}
\usepackage{amsmath}
\usepackage[version=3]{mhchem}
\usepackage{anysize}
\usepackage{adjustbox}
\usepackage{mathtools}
\usepackage{arydshln}
\usepackage{color}
\usepackage[english]{babel}
\usepackage{amsthm}
\usepackage{framed}
\usepackage{tikz}
\usepackage{enumerate}

\usetikzlibrary{arrows,decorations.pathmorphing,backgrounds,positioning,fit,matrix}

\usepackage[normalem]{ulem}

\newtheorem{theorem}{Theorem}[section]
\newtheorem{lemma}[theorem]{Lemma}

\newtheorem{proposition}[theorem]{Proposition}
\newtheorem{prop}[theorem]{Proposition}
\newtheorem{corollary}[theorem]{Corollary}

\theoremstyle{definition}
\newtheorem{defn}[theorem]{Definition}
\newtheorem{definition}[theorem]{Definition}

\newtheorem{example}[theorem]{Example}

\newtheorem{remark}[theorem]{Remark}

\renewcommand{\k}{{\kappa}}

\DeclareMathOperator{\im}{Im}

\def\la{\leftarrow}
\def\lra{\leftrightarrows}

\newcommand{\lradot}{%
  \mathrel{\ooalign{\hfil$\vcenter{
   \hbox{$\mkern3mu\scriptscriptstyle\bullet$}}$\hfil\cr$\longleftrightarrow$\cr}
  }%
}

\newcommand{\scc}{\mathbf{S}}
\def\SS{\mathcal S}
\def\CC{\mathcal C}
\def\RR{\mathcal R}
\newcommand{\invtPoly}{\mathcal{P}}

\newcommand{\St}{{S}}
\newcommand{\Rnn}{\mathbb{R}_{\geq 0}}

\begin{document}
\title{Absolute concentration robustness in networks  with low-dimensional stoichiometric subspace}
\author{Nicolette Meshkat\footnote{Santa Clara University}, Anne Shiu\footnote{Texas A\&M University}, and Angelica Torres\footnote{KTH Royal Institute of Technology}}
\date{\today}
\maketitle  

\begin{abstract}
A reaction system exhibits ``absolute concentration robustness'' (ACR) in some species if the positive steady-state value of that species does not depend on initial conditions.  
Mathematically, this means that the positive part of the variety of the steady-state ideal lies entirely in a hyperplane of the form $x_i=c$, for some $c>0$. 
Deciding whether a given reaction system -- or those arising from some reaction network -- 
exhibits ACR is difficult in general, but here we show that for many simple networks, assessing ACR is straightforward.  Indeed, our criteria for ACR can be performed by simply inspecting a network or its standard embedding into Euclidean space. 
Our main results pertain to networks with many conservation laws, so that all reactions are parallel to one other.  Such ``one-dimensional'' networks include those networks having only one species.
We also 
consider networks with only two reactions, and show that ACR is characterized by 
a well-known criterion of Shinar and Feinberg.  
Finally, up to some natural
ACR-preserving 
operations -- 
relabeling species, lengthening a reaction, and so on 
--
only three families of networks with two reactions and two species have ACR.
Our results are proven using algebraic and combinatorial techniques.  

\vspace{0.1in}

\end{abstract}

\section{Introduction} \label{sec:intro}
The concept of {\em absolute concentration robustness} (ACR) was introduced by Shinar and Feinberg in their investigations into how biochemical systems maintain their function despite changes in the environment~\cite{shinar2010structural}.  
 A biochemical system exhibits ACR in some species $X_i$ if for every positive steady state $x$, regardless of initial conditions, the value of $x_i$ is the same.  
Thus, a system with ACR maintains species $X_i$ at a constant level, even amid environmental fluctuations. Mathematically, ACR occurs when the positive part of the variety of the steady-state ideal lies entirely in a hyperplane of the form $x_i=c$, for some $c>0$.  ACR is therefore rare.

Many of the known results on ACR are sufficient conditions for ACR, which arise from the reaction-network structure~\cite{shinar2010structural}, through the analyses of elementary modes~ \cite{neigenfind2011biochemical,neigenfind2013relation}, or by network ``translation'' (a way of understanding the steady states of one network through those of a closely related network)~\cite{tonello2017network}.  Algebraic methods also have been harnessed~\cite{karp2012complex, beatriz-elisenda, Mercedes}, and several important biochemical systems have been analyzed in the context of ACR~\cite{dexter2015,dexter2013dimerization}. 
There has also been interest from control theory~\cite{hidden, controller}.
Finally, ACR has also been considered in stochastic, rather than deterministic, models \cite{anderson2019discrepancies, Anderson2016,  A-E-J, Enciso2016}.

A starting point of this work is the first sufficient condition for ACR, from the original article of Shinar and Feinberg~\cite{shinar2010structural}. Here we show that their criterion completely characterizes ACR in networks having only two reactions (and at least two species), under mass-action kinetics, as long as a mild nondegeneracy condition is met.  (This condition ensures that a network like $\{B \lra A+B\}$ is viewed as having only one species, $A$, with $B$ acting only as a catalyst.) Therefore, ACR is easy to check for two-reaction networks (see Theorem~\ref{thm:2-rxn}), even though this is a difficult task for general networks.  

It turns out that 
for other classes of networks ACR is again easy to assess.  One such class is formed by networks with only one species -- or, more generally, those with many conservation laws so that all reaction vectors are scalar multiples of each other.  For these networks, ACR is checked simply from inspecting the directions of the reaction arrows in an embedding of the network in Euclidean space. 
See Proposition~\ref{prop:1-species} and Theorem~\ref{thm:1-dim-stable}.

Finally, we specialize our results to the case of networks with exactly 
two reactions and 
two species (Corollary~\ref{cor:2-rxn-2-species}). 
We prove that up to certain network operations which preserve ACR -- such as relabeling species and lengthening a reaction -- there are exactly three families of networks with ACR:
\begin{enumerate}
\item $ \{0 \lra mA \}$, for $m \geq 1$, 
\item $\{ A \to 2A, ~ A+ nB \to nB \} $, for $n \geq 1$, and
\item $\{ B \to A, ~ pA+B \to (p-1)A+2B \}$, for $p \geq 1$.
\end{enumerate}
In the third family,
the $p=1$ network 
goes back to the original article 
on ACR~\cite{shinar2010structural}.

It might seem that our results, which apply to a limited class of networks, are only moderately interesting. 
However, the original interest in ACR pertained to small networks and the capacity of ACR in such networks to confer robustness even when situated inside a larger network~\cite{shinar2010structural}.  
Indeed, such results on how ACR -- and even the steady-state value -- are maintained when located within larger networks, are proved in a recent article of Cappelletti, Gupta, and Khammash~\cite{hidden}.  
Therefore, having a database and (as much as possible) a classification of small networks having ACR will aid in analyzing ACR -- and robustness more generally -- in applications.  

The outline of our work is as follows.  We begin with a Background in Section~\ref{sec:background}, and then introduce network operations in Section~\ref{sec:operations}.  
In Section~\ref{sec:1-species}, we present our results on networks with only one species.  
Next, we generalize those results to handle one-dimensional networks with any number of species in Section~\ref{sec:main-results}.  
In Section~\ref{sec:2-species-2-rxns}, 
we prove our results on 
networks with two reactions.    We end with a Discussion in Section~\ref{sec:discussion}.

\section{Background} \label{sec:background} 
This section recalls pertinent definitions and prior results related to reaction networks and ACR.

\subsection{Reaction networks} \label{sec:CRN} 
We start by defining a reaction network, and then we define the system of differential equations modelling the dynamics of networks under mass-action kinetics. 

\begin{definition}
A {{\em reaction network}} ${G} =(\SS,\CC,\RR)$ consists of three finite sets:
\begin{enumerate}
    \item a set of {{\em species}} ${\SS} := \{A_1,A_2,\dots, A_s\}$; 
    \item a set of {{\em complexes}} ${\CC} := \{y_1, y_2, \dots, y_p\}$, consisting of finite nonnegative-integer combinations of the species; and 
    \item a set of {{\em reactions}} ${\RR} \subseteq  (\CC \times \CC) \smallsetminus \{ (y_i,y_i) \mid y_i \in \CC\}$, which are ordered pairs of complexes, excluding diagonal pairs, 
    such that every complex takes part in at least one reaction. 
\end{enumerate}
\end{definition}

\noindent
Throughout our work, 
$s$ and $r$ denote the numbers of
species and reactions, respectively.  

It is customary to write a reaction $(y_i,y_j)$ as $y_i \to y_j$. The complexes $y_i$ and $y_j$ are called the \emph{reactant} and \emph{product} respectively.
Also, a reaction $y_i \to y_j$ is {{\em reversible}} if its {\em reverse reaction} $y_j \to y_i$ is also in $\RR$, and we denote such a pair by $y_i \rightleftharpoons y_j$.

It will sometimes be convenient to write reactions as $y\to y'$ (rather than $y_i \to y_j$) so that indices correspond to the species rather than the complexes.  This usage should be clear from context. Also, our examples will involve only a few species, and so we will write $A, B, C, \dots$ for the species rather than $A_1, A_2, A_3 \dots $.

\begin{example} \label{ex:generalized-shinar--Feinberg}
Let $n \geq 1$.  
The following {\em generalized Shinar--Feinberg network} has 2 species, 4 complexes, and 2 reactions:
\begin{align} \label{eq:generalized-shinar--Feinberg-network}
\{ B \to A, ~ nA+B \to (n-1)A+2B \}~.    
\end{align}
The $n=1$ version of this network was analyzed by Shinar and Feinberg~\cite{shinar2010structural}.  
\end{example}

Our next task is to explain how a network gives rise to a system of ODEs.  We begin by writing the $i$-th complex as $y_{i1} A_1 + y_{i2} A_2 + \cdots + y_{is}A_s$, 
where each $y_{ij} \in \mathbb{Z}_{\geq 0}$ 
is called the {{\em stoichiometric coefficient}} of $A_j$, for $j=1,2,\dots,s$.
Next, ${x_1},{x_2},\ldots,{x_s}$ represent the
concentrations of the $s$ species,
which we view as functions ${x_i(t)}$ of time $t$, and we define the monomial 
${\mathbf{x}^{y_i}} \,\,\, := \,\,\, x_1^{y_{i1}} x_2^{y_{i2}} \cdots  x_s^{y_{is}} $. 

A reaction $y_i \to y_j$, from the $i$-th complex to the $j$-th
complex, defines the {{\em reaction vector}}
 ${y_j-y_i}$, which encodes the
net change in each species 
resulting from the reaction. The {{\em stoichiometric matrix}} 
${\Gamma}$ is the $s \times r$ matrix whose $k$-th column 
is the reaction vector of the $k$-th reaction.  
Under mass-action kinetics, each reaction $y_i \to y_j$ comes with a
{{\em rate constant}} ${\kappa_{ij}}$,
which is a positive real parameter.

A {{\em mass-action system}}, 
denoted by $(G, \kappa)$, 
is the 
dynamical system that arises, via mass-action kinetics, from a chemical reaction
network $G=(\SS, \CC, \RR)$ and a choice of rate constants $\kappa = (\kappa_{ij}) \in
\mathbb{R}^{r}_{>0}$ (recall that $r$ is the number of
reactions), as follows:
\begin{align} \label{eq:ODE-mass-action}
\frac{d\mathbf{x}}{dt} \quad = \quad \sum_{ y_i \to y_j~ \in \RR} \kappa_{ij} \mathbf{x}^{y_i}(y_j - y_i) \quad =: \quad {f_{\kappa}(\mathbf{x})}~.
\end{align}

Viewing the rate constants as a vector of parameters 
$\kappa = (\kappa_{ij}) \in
\mathbb{R}^{r}_{>0}$, 
we have polynomials $f_{\kappa,i} \in \mathbb Q[\kappa,x]$, for $i=1,2, \dots, s$.  For simplicity, we often write $f_i$ instead of $f_{\kappa,i}$.

The {{\em stoichiometric subspace}}, 
  ${\St} := {\rm span} \left( \{ y_j-y_i \mid  y_i \to y_j~ {\rm is~in~} \RR \} \right)$, 
 is the vector subspace of
$\mathbb{R}^s$ spanned by all reaction vectors
$y_j-y_i$.  
Thus, $\St = \im(\Gamma)$, where $\Gamma$ is the stoichiometric matrix defined earlier.

Let $d=s-{\rm rank}(\Gamma)$.  
A {\em conservation-law matrix} of $G$, denoted by~$W$, is a row-reduced $d\times s$-matrix whose rows form a basis of 
the orthogonal complement of $S$. 
A trajectory $x(t)$ that starts at a 
    positive vector $x(0)=x^0 \in
    \mathbb{R}^s_{> 0}$ 
remains, for all positive time,
 in the following {\em stoichiometric compatibility class} with respect to the {total-constant vector} $c\coloneqq W x^0 \in {\mathbb R}^d$  (where $W$ is a conservation-law matrix): 
    \begin{align} \label{eqn:invtPoly}
    \scc_c~\coloneqq~ \{x\in {\mathbb R}_{\geq 0}^s \mid Wx=c\}~.
    \end{align}

\begin{example}[Example~\ref{ex:generalized-shinar--Feinberg}, continued] \label{ex:generalized-shinar--Feinberg-cont}
Returning to the
 generalized Shinar--Feinberg network~\eqref{eq:generalized-shinar--Feinberg-network}, 
 the resulting ODEs~\eqref{eq:ODE-mass-action} are as follows:
 \begin{align} \label{eq:ode-shinar--Feinberg}
    \frac{dx_A}{dt} ~&=~ \kappa_1 x_B - \kappa_2 x_A^n x_B \\
    \frac{dx_B}{dt} ~&=~ -\kappa_1 x_B + \kappa_2 x_A^n x_B~, \notag
 \end{align}
 where $\kappa_1$ and $\kappa_2$ are the rate constants for the first and second reactions in~\eqref{eq:generalized-shinar--Feinberg-network} respectively. The stoichiometric subspace is spanned by $(1,-1)^T$ and so is one-dimensional.  The stoichiometric compatibility classes are the following line segments, for $c \in \mathbb{R}_{>0}$:
 \begin{align} \label{eq:example-compatibility-class}
 \scc_c ~=~ \{ (x_A,~x_B) \in \mathbb{R}^{2}_{\geq 0} \mid x_A+x_B=c\}~.
 \end{align}
\end{example}

Many of the results in this work pertain to networks, like the generalized Shinar--Feinberg networks, 
in which the stoichiometric subspace is one-dimensional.  Another class that will appear often 
is formed by networks in which only a single species appears in all the reactions.  Accordingly, we give names to these networks, as follows.

\begin{definition} \label{def:1-d-1-species}
Let ${G} =(\SS,\CC,\RR)$ be a reaction network.  
\begin{enumerate}
    \item $G$ is {\em one-dimensional} if its stoichiometric subspace is one-dimensional.  
    \item $G$ is a {\em one-species network} if there exists a species $A_i \in \SS$ such that every reaction involves only $A_i$, that is, $(y,y') \in \RR$ implies that $y_j=y'_j=0$ for all species $A_j \in \SS \smallsetminus \{A_i\} $.
\end{enumerate}
\end{definition}

Every reaction network ${G} =(\SS,\CC,\RR)$ with $|\SS|=1$ is a one-species network.  However, our definition allows larger species sets, 
as long as the extra species do not take part in reactions, in order to accommodate our definitions of network operations in Section~\ref{sec:operations}.

\begin{remark} \label{rem:1-d-net-in-lit}
One-dimensional networks have been analyzed recently in terms of multistationarity and multistability~\cite{tang-xu} and in the stochastic setting~\cite{wiuf2020classification, xu2020criteria}.
\end{remark}

\begin{remark} \label{rem:1-species}
Our definition of one-species networks is more general than how the term is used in~\cite{joshi-shiu-I}, but this should not cause any confusion.
\end{remark}

As mentioned above, one-dimensional and one-species networks form two classes of networks that play a key role in our results.  We also consider a third class, as follows.

\begin{definition} \label{def:cat-only}
A species $A_i \in \SS$ is a {\em catalyst-only species} of a reaction network $G$ if, for all reactions $y \to y'$ of $G$, we have $y_i=y_i'$.
\end{definition}

\begin{example}[Degenerate-ACR network] \label{ex:daniele}
Let $n \geq 1$.  
Consider the following network, which was shown to us by Daniele Cappelletti:
\begin{align} \label{eq:daniele}
\{ A \to 2A, ~ A+ nB \to nB \}~.    
\end{align}
We call this network the {\em degenerate-ACR network} (we will see later in this section that this network has ACR but all steady states are degenerate). 
The species $B$ is a catalyst-only species, and the  ODEs~\eqref{eq:ODE-mass-action} are as follows:
 \begin{align} \label{eq:daniele-odes}
    \frac{dx_A}{dt} ~&=~ \kappa_1 x_A - \kappa_2 x_A x_B^n \\
    \frac{dx_B}{dt} ~&=~ 0
    \notag 
 \end{align}
 where $\kappa_1$ and $\kappa_2$ are the rate constants for the first and second reactions in~\eqref{eq:daniele}, respectively.
\end{example}

In Example~\ref{ex:daniele}, we saw that the catalyst-only species $B$ gave rise to an ODE with zero right-hand side.  It is easy to see that this generalizes, as follows.

\begin{lemma} \label{lem:cat-only-species}
 Let $A_i$ be a species of a reaction network $G$.  
 Then $A_i$ is a catalyst-only species of $G$
 if and only if 
    in the mass-action ODEs~\eqref{eq:ODE-mass-action}, 
    for all choices of positive rate constants, 
 we have 
 $\frac{dx_i}{dt}=0$.
\end{lemma}

\subsection{Steady states} \label{sec:steady-state}
For a mass-action system,
a {{\em steady state}} is a nonnegative concentration vector ${\mathbf{x}^*} \in \Rnn^s$ at which the right-hand side of the ODEs~\eqref{eq:ODE-mass-action}  vanishes: $f_{\kappa} (\mathbf{x}^*) = 0$.  
A steady state $\mathbf{x}^*$ is {{\em nondegenerate}} if ${\rm Im}\left( df_{\kappa} (\mathbf{x}^*)|_{S} \right) = \St$, where ${df_{\kappa}(\mathbf{x}^*)}$ is the Jacobian matrix of $f_{\kappa}$ evaluated at $\mathbf{x}^*$.  
A nondegenerate steady state is {\em exponentially stable} (or, for brevity, {\em stable}) if each of the 
$\dim(\St)$ 
nonzero eigenvalues of $df_{\kappa}(x^*)$ has negative real part. If one of these eigenvalues has positive real part, then $x^*$ is {\em unstable}. 

Our main interest in this work is in {{\em positive steady states}} $\mathbf{x} ^* \in \mathbb{R}^s_{> 0}$.  
The set of all positive steady states of a mass-action system is 
the {\em positive steady-state locus}.

\begin{definition}~ \label{def:mss}
A network {\em admits a positive steady state}  
(respectively, is {{\em multistationary}})
if there exist positive rate constants $(\kappa_{ij})\in\mathbb{R}^m_{>0}$
such that, for the corresponding dynamical system~\eqref{eq:ODE-mass-action}, there is some stoichiometric compatibility class~\eqref{eqn:invtPoly} having at least one (respectively, at least two) positive steady state(s). 
\end{definition}

\begin{example}[Example~\ref{ex:generalized-shinar--Feinberg-cont}, continued] \label{ex:generalized-shinar--Feinberg-3}
From the ODEs~\eqref{eq:ode-shinar--Feinberg} of the
 generalized Shinar--Feinberg network, 
 we see that the positive steady-state locus is the set 
 $\{  (\sqrt[n]{\kappa_1/\kappa_2},~ b) \mid b \in \mathbb{R}_{>0} \}$.
This set intersects each stoichiometric compatibility class~\eqref{eq:example-compatibility-class} at most once. Thus, the
 generalized Shinar--Feinberg network is non-multistationary.  
\end{example}

\begin{example}[Example~\ref{ex:daniele}, continued]
\label{ex:daniele-2}
From the ODEs~\eqref{eq:daniele-odes} of the degenerate-ACR network, 
we see that the stoichiometric subspace is spanned by $(1,0)^{T}$ and the stoichiometric compatibility classes are the horizontal, closed half-rays $\invtPoly_c = \{(x_A, c) \mid x_A \in \mathbb{R}_{\geq 0} \}$, for $c \in \mathbb{R}_{>0}$.  Also, the positive steady-state locus is the interior of a single compatibility class,  $\invtPoly_{c^*}$, where $c^*= \sqrt[n]{\kappa_1/\kappa_2}$. It is straightforward to check that all of these positive steady states are degenerate.
\end{example}

Some of our results will pertain to networks with two reactions.  The following lemma states that such networks, if they admit a positive steady state, must be one-dimensional.

\begin{lemma} \label{lem:2-rxn-steady-state-implies-1-d}
Let $G$ be a reaction network with exactly two reactions.  If $G$ admits a positive steady state, then $G$ is one-dimensional.
\end{lemma}

\begin{proof}
This follows easily from~\cite[Lemma~4.1]{joshi-shiu-I} (or a straightforward calculation). 
\end{proof}

\begin{remark}
The converse of Lemma~\ref{lem:2-rxn-steady-state-implies-1-d} is false: 
the network $\{ 0 \to A \to 2A\}$ consists of two reactions and is one-dimensional, but does not admit a positive steady state.
\end{remark}

\subsection{Deficiency} \label{sec:deficiency}
The  deficiency $\delta$ is a nonnegative integer associated with each network (see Definition~\ref{def:deficiency} below). It is an important invariant, with a close relationship to steady states and their stability~\cite{feinberg-def0}.  

To define deficiency, we need some terminology. 
By representing each reaction $(y_i,y_j)\in\mathcal{R}$ as $y_i\rightarrow y_j$,
we obtain a {\em reaction graph} $\mathcal{G}$ 
in which the vertices correspond to the complexes, and the (directed) edges correspond to the reactions. The \emph{linkage classes} of a network are the connected components of $\mathcal{G}$. 
The {\em terminal strong linkage classes} are the strongly connected components of $\mathcal{G}$ such that no reaction points out of the component (that is, there is no reaction $y \to y'$ with $y$ in the component but not $y'$).
A complex is \emph{terminal} if it belongs to a terminal strong linkage class; otherwise, it is \emph{nonterminal}.

\begin{definition}\label{def:deficiency}
    The \emph{deficiency} of a reaction network is 
    \[\delta := p-\ell - \dim(S)~,\]
    where $p$ is the number of complexes, $\ell$ is the number of linkage classes, and $S$ is the stoichiometric subspace. 
\end{definition}


\begin{example}[Example~\ref{ex:generalized-shinar--Feinberg-3}, continued] \label{ex:generalized-shinar--Feinberg-4}
The generalized Shinar--Feinberg network has 
4 complexes and 2 linkage classes.  We saw earlier that the network is one-dimensional.  
Hence, the deficiency is $4-2-1=1$.  
\end{example}

\subsection{Geometric diagrams for networks} \label{sec:geometric-diagram}
 We recall how to represent a network geometrically, through the reaction diagram~\cite{ProjArg}.
 
 \begin{definition} \label{def:reaction-diagram}
 Let $G$ be a network with $s$ species.
  The {\em reaction diagram} of $G$ is the realization, in $\mathbb{R}^s$, of its reaction graph in which each reaction $y_1 A_1+ y_2 A_2 + \dots y_s A_s \to z_1 A_1+ z_2 A_2 + \dots z_s A_s$ 
  is depicted as the arrow from $(y_1,y_2,\dots, y_s) $ to $(z_1, z_2, \dots, z_s)$.
 \end{definition}
 
 \begin{remark}
 Reaction diagrams are also called ``Euclidean embedded graphs''~\cite{CraciunSIAGA}.
 \end{remark}

\begin{example}[Example~\ref{ex:generalized-shinar--Feinberg-4}, continued] \label{ex:generalized-shinar--Feinberg-5}
    The reaction diagram of the generalized Shinar--Feinberg network is shown 
    here:
    \begin{center}
    	\begin{tikzpicture}[scale=.75]
    	\draw (-1,0) -- (3.25, 0);
    	\draw (0,-.5) -- (0, 3.5);
    	\draw [->] (0,1) -- (1, 0.1);
    	\draw [->] (4,1) -- (3, 2);
        \node [left] at (0, 1) {$B$};
        \node [right] at (1, 0.3) {$A$};
        \node [right] at (4,1) {$nA+B$};
        \node [above] at (3, 2) {$(n-1)A+2B$};
    	\end{tikzpicture}
    \end{center}
    \end{example}

\begin{example} \label{ex:1}
The reaction diagram of the network 
$\{ 3A + 5B \to A + 6B~, ~A + 3B \to 3A + B \}$
is presented in Figure~\ref{fig:geom_rep}.
\end{example}

\begin{figure}
    \centering
    \includegraphics[scale=0.6]{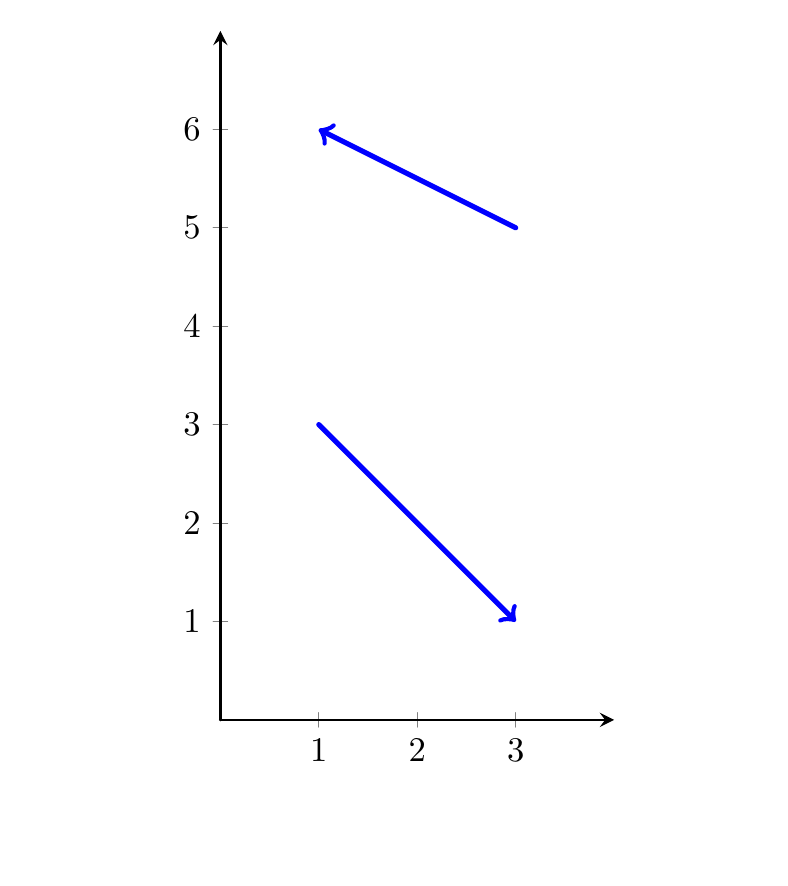}
    \caption{Reaction diagram 
    of the network in Example~\ref{ex:1}.}
    \label{fig:geom_rep}
\end{figure}

\subsection{Embedded networks and arrow diagrams} \label{sec:embedded-network}
Here we recall two concepts -- embedded networks and arrow diagrams -- which will be useful when we analyze one-dimensional networks~\cite{joshi-shiu-I}.

\begin{definition} \label{def:restrict}
The {\em restriction} of a set of reactions $\RR$ to a set of species $\SS$, denoted by $\RR |_{\SS}$, is the set obtained from $\RR$ by (1) setting to 0 the stoichiometric coefficients of all species
not in $\SS$, and then (2) discarding any {\em trivial reactions} (reactions of the form $\sum m_i A_i \to \sum m_i A_i$, that is, in which the reactant and product complexes are equal) and keeping only one copy of any duplicate reactions that arise.
\end{definition}
\begin{definition} \label{def:emb}
The {\em embedded network} ${N}$ of a network $G = (\SS,\CC,\RR)$ 
 obtained by removing a set of reactions $\RR' \subseteq \RR$ and 
 a set of species $\{X_i\}_{i \in I} \subseteq \SS$ is
\[
{N} \ := \ \left( \SS|_{\CC|_{\RR_N}},~ \CC|_{\RR_N}, \ \RR_N := \left(\RR \smallsetminus   \RR' \right)|_{\SS \smallsetminus \{X_i\}_{i \in I} }\right)~,
\]
where
$ \CC|_{\RR_N}$ denotes the set of complexes of the set of reactions $\RR_N$, and 
$\SS|_{\CC|_{\RR_N}}$ denotes the set of species in the set of complexes 
$\CC|_{\RR_N}$.
\end{definition}

\begin{remark}[Embedded networks as projections] \label{rem:projection}
All embedded networks considered in this work are obtained by removing all species except one.  Accordingly, such embedding networks can be viewed as projections of the original network 
(or its reaction diagram) 
to the axis corresponding to the single species.  See, for instance, Example~\ref{ex:generalized-shinar--Feinberg-6} below or (in a later section) Figure~\ref{fig:ACR2spe2rxn}.
\end{remark}

\begin{definition} \label{def:arrow-diagram}
Let $G$ be a one-species reaction network with at least one reaction. 
Let $A$ denote the unique species appearing in all reactions, so that each reaction of $G$  has the form $aA \to bA$, where $a,b \ge 0$ and $a \ne b$. Let $m$ be the number of (distinct) reactant complexes, and let $a_1< a_2 < \ldots < a_m$ be their stoichiometric coefficients. The {\em arrow diagram of $G$}, denoted $\rho = (\rho_1, \rho_2, \ldots , \rho_m)$, is the element of $\{\to , \la, \lradot \}^m$ defined by:
\begin{equation*}
 \rho_i~=~ 
 \left\lbrace\begin{array}{ll}
   \to & \text{if for all reactions $a_iA \to bA$ in $G$, it is the case that $b > a_i$} \\
   \la & \text{if for all reactions $a_iA \to bA$ in $G$, it is the case that $b < a_i$} \\
   \lradot & \text{otherwise.}
 \end{array}\right.
\end{equation*}
\end{definition}

\begin{example}[Example~\ref{ex:generalized-shinar--Feinberg-5}, continued] \label{ex:generalized-shinar--Feinberg-6}
For the generalized Shinar--Feinberg network, $\{ B \to A, ~ nA+B \to (n-1)A+2B \}$, 
the embedded network obtained by removing species~$A$ is $\{ 0 \leftarrow B \to 2B\}$, which has arrow diagram $(\lradot)$.  On the other hand, removing $B$ yields $\{0 \to A~,~ (n-1)A \leftarrow nA\}$, with arrow diagram $(\to, \leftarrow)$.
\end{example}

\begin{remark} \label{rem:arrow-diagram-literature}
Arrow diagrams have been used to analyze properties of one-species networks, specifically, multistationarity~\cite{JS}, multistability~\cite{tang-xu}, and the ``mixed volume''~\cite{mv-small-networks}.
\end{remark}

\subsection{Absolute concentration robustness} \label{sec:ACR}
Here we recall the definition of absolute concentration robustness (ACR), and introduce stable ACR.  
Stable ACR is the type of ACR that is most relevant in applications, as in this case the steady states are robust to small perturbations.  

 \begin{definition} \label{def:ACR-system}
Let $A_i$ be a species of a reaction network $G$.
	\begin{enumerate}[(i)]
	\item A mass-action system $( G, \kappa)$ has \emph{absolute concentration robustness} (ACR) in $A_i$ if the value of $x_i$ in every positive steady state $x$ of the system, is the same.  
This value of $x_i$ is the {\em ACR value}.
    \item A mass-action system $( G, \kappa)$ has {\em stable ACR} 
    if $( G, \kappa)$ has ACR in $A_i$ and additionally every positive steady state of the system is stable. 
    \end{enumerate}
\end{definition} 

\begin{remark} \label{rem:no-pos-steady-states}
A mass-action system with no positive steady states trivially has ACR (and also stable 
ACR); this is called ``vacuous ACR'' in~\cite{Mercedes}.  
Mass-action systems with vacuous ACR include those systems arising from networks with only one reaction. 
However, in applications, 
we are not interested in vacuous ACR.
Therefore, our main results will consider only networks that admit positive steady states.
\end{remark}

\begin{remark} \label{rem:local-acr}
A weaker form of ACR, in which some species has only finitely many positive steady-state values, was analyzed recently by Pascual-Escudero and Feliu~\cite{beatriz-elisenda}.
\end{remark}

In this work, we are interested in ACR at the level of networks, rather than at the level of one of its mass-action systems.  Accordingly, we introduce the following definition.
 
 \begin{definition} \label{def:ACR-network}
A network $ G$ with $m$ reactions and species $A_1, A_2, \dots, A_s$ has {\em ACR in}~$A_i$ if $( G, \kappa)$ has ACR in $A_i$ for every
		choice of $\kappa \in \mathbb R_{>0}^m$. Similarly, $G$ may have {\em stable ACR} 
		in $A_i$.
Also, $G$ has {\em ACR} (or {\em stable ACR}) if $G$ has ACR
(or stable ACR)
in some species $A_i$.
\end{definition}

\begin{remark} \label{rem:acr-network-could-be-vacuous}
A network $G$ that has ACR (or stable ACR) could have vacuous ACR (as in Remark~\ref{rem:no-pos-steady-states}) for some or even all of its mass-action systems.  However, as mentioned earlier, we usually focus on networks that admit a positive steady state, so that vacuous ACR is permitted for some but not all of the resulting mass-action systems.
\end{remark}

\begin{example} \label{ex:1-species-4-reactions}
Consider the following network $G$, which has only one species: 
    \begin{align*}
    \ce{0 <-[\kappa_1] A1 } \quad 
    \ce{A1 ->[\kappa_2] 2A1} \quad
    \ce{A1 <-[\kappa_3] 2A1} \quad
    \ce{3A1 ->[\kappa_4] 4A1} 
    \end{align*}
It is straightforward to check that when 
$\kappa=(\kappa_1,\kappa_2,\kappa_3,\kappa_4)=(3,1,1,1)$, 
the mass-action system $(G,\kappa)$ has a unique positive steady state, namely, $x_{A_1}^*=2$, and hence has ACR (in species ${A_1}$).  
On the other hand, when 
$(\kappa_1,\kappa_2,\kappa_3,\kappa_4) = (1,3,3,1)$, 
there are two positive steady states ($x_{A_1}^*=1$ and $x_{A_1}^*=2$), and hence no ACR.  Thus, $G$ does not have ACR (but some of its mass-action systems do).
\end{example}

In Example~\ref{ex:1-species-4-reactions}, we saw a one-species network with ACR for some but not all of its mass-action systems; a more interesting example, with two species, was described in the thesis of P\'erez Mill\'an~\cite[Example~6.5.3]{Mercedes}.

The following sufficient condition for ACR  is due to Shinar and Feinberg~\cite{shinar2010structural}:

\begin{proposition} 
\label{prop:shinar--Feinberg-criterion}
Suppose that a mass-action system $(G,\kappa)$ has a positive steady state, 
the deficiency of $G$ is 1, 
and $G$ has two non-terminal complexes that differ in only some species $A_i$. Then the system $(G,\kappa)$ has ACR in $A_i$.
\end{proposition}

In Proposition~\ref{prop:shinar--Feinberg-criterion}, the assumption of having a positive steady state is not required (recall Remark~\ref{rem:no-pos-steady-states}), so we can restate the criterion at the level of networks as follows.

\begin{proposition} [Shinar--Feinberg criterion]
\label{prop:shinar--Feinberg-criterion-network}
If $G$ is a reaction network that has deficiency 1 and has two non-terminal complexes that differ in only some species $A_i$, 
then $G$ has ACR in $A_i$.
\end{proposition}

\begin{example}[Example~\ref{ex:generalized-shinar--Feinberg-6}, continued] \label{ex:gen-canonical-network}
Returning to the generalized Shinar--Feinberg network, we saw (in Example~\ref{ex:generalized-shinar--Feinberg-cont}) that every positive steady state $(x_A, x_B)$ satisfies $x_A=\sqrt[n]{\kappa_1/\kappa_2}$.  Thus, the network has ACR in the species $A$.  
Another way to detect ACR in this network is through the Shinar--Feinberg criterion (Proposition~\ref{prop:shinar--Feinberg-criterion-network}).
Indeed, 
$B$ and $nA+B$ are non-terminal complexes that differ only in species $A$, and we saw in Example~\ref{ex:generalized-shinar--Feinberg-4} that this network has deficiency 1. 
\end{example}

\begin{example}[Example~\ref{ex:daniele-2}, continued]
\label{ex:daniele-3}
The degenerate-ACR network has ACR in species $B$: we saw in 
Example~\ref{ex:daniele-2} that every positive steady state $(x_A,x_B)$ satisfies $x_B = \sqrt[n]{\kappa_1/\kappa_2}$. Notably, the Shinar--Feinberg criterion also applies. 
\end{example}

Next, we see that the Shinar--Feinberg criterion does not cover all cases of ACR.
\begin{example}[Shinar--Feinberg criterion is not necessary for ACR] \label{ex:acr-but-not-shinar--Feinberg}
Consider the one-species network 
$\{0 \lra A \}$.  
It is straightforward to check by hand (or by applying Proposition~\ref{prop:1-species} in a later section) that this network has ACR.
Nonetheless, the Shinar--Feinberg criterion (Proposition~\ref{prop:shinar--Feinberg-criterion-network}) does not apply: the deficiency is $\delta = 2-1-1=0$.  
More generally, the networks $\{0 \lra nA \}$, for $n \geq 1$, also have ACR and deficiency 0.
Another network with ACR that is not detected by the Shinar--Feinberg criterion is $\{0 \to A \to 2A~,~ 3A \leftarrow 4A\}$, which has deficiency $\delta = 5-2-1=2$.  
\end{example}

In fact, there are no known necessary conditions for ACR (for general networks) that are easy to check.  We clarify this claim in light of a recent article which gave a purported necessary condition~\cite{eloundou2016network}.  The setup for this condition involves a network $G$, one of its species $A_i$, and the embedded network $H$ obtained by removing species $A_i$ from $G$.  What is claimed in~\cite{eloundou2016network} is that, if the mass-action system $(G,\kappa)$ has a positive steady state and the mass-action system $(H, \kappa(H))$ also has a positive steady state 
(here, $\kappa(H)$ indicates that 
the rate constant of a
reaction $y \to y'$ of $H$ is the sum of the rate constants of those reactions of $G$ that project to $y\to y'$ -- this definition is inferred from but not explicitly given in~\cite{eloundou2016network}), then a necessary condition for $G$ to have ACR in $A_i$ is that $G$ and $H$ have the same deficiency.  This claim, however, is overstated, as the following example shows.

\begin{example} \label{ex:eloundou}
Consider the network $G$ given by 
$\{B \to A,~ A+B \to 2B,~ 2A+B \to A+2B\}$.
It is straightforward to check (for instance, apply Theorem~\ref{thm:1-dim-stable}) that for all rate constants $\kappa$, the mass-action system $(G,\kappa)$ has a positive steady state.  Next, the embedded network $H$ obtained by removing $A$ is $\{0 \la B \to 2B\}$, which admits a positive steady state when the two rate constants are equal.  The deficiency of $G$ is $6-3-1=2$, while the deficiency of $H$ is $3-1-1=1$.  Nevertheless, $G$ has ACR in $A$ (this can be checked by hand, or we will see this from Theorem~\ref{thm:1-dim-stable}).
\end{example}

\section{Network operations} \label{sec:operations}

This section introduces several operations on reaction networks (Definition~\ref{def:operations}), 
which correspond to the following geometric operations on the reaction diagram: 
(1) 
reflecting the diagram across some 
diagonal hyperplanes $x_i=x_j$, 
(2) 
translating the diagram, 
(3) 
changing the length of one arrow, 
(4) 
replacing an arrow with two arrows of different lengths
and 
(5) 
applying certain rotations to each arrow in the diagram
(see Lemma~\ref{lem:translate}).
Our main result states that all of these operations preserve ACR (Theorem~\ref{thm:operations-ACR}).  
We also clarify which of these operations preserve 
deficiency, the dimension of the stoichiometric subspace, and the steady state locus.  These results are summarized in Table~\ref{tab:summary-operations}.

\begin{table}[htb]
\begin{tabular}{lcccc}
\hline
Operation            & Deficiency? & Dim? & Positive steady-state locus?          & ACR? \\
\hline
Relabel species      & Yes         & Yes                                   & No (reflects across hyperplanes) & Yes  \\
Translate            & Yes         & Yes                                   & Yes                                   & Yes  \\
Stretch a reaction   & No          & Yes                                   & No*                                   & Yes  \\
Duplicate a reaction & No          & Yes                                   & No*                                   & Yes  \\
Partial scaling      & No          & Yes                                   & Yes                                   & Yes 
\\
\hline
\end{tabular}
\caption{Five operations on networks (from Definition~\ref{def:operations}) and whether they preserve the properties of deficiency, dimension of the stoichiometric subspace, the positive steady-state locus, and absolute concentration robustness (ACR).
For details, see
Propositions~\ref{prop:4-opers-preserve-dim},~\ref{prop:operations-preserve-deficiency}, and~\ref{prop:operations-preserve}; 
Theorem~\ref{thm:operations-ACR}; 
and Examples~\ref{ex:stretch-def}--\ref{ex:partial-scale-def}. 
Also, ``No*'' indicates that an appropriate transformation of the reaction rate constants is required so that the steady-state locus is preserved; see Lemma~\ref{lem:stretching} and the
proof of Theorem~\ref{thm:operations-ACR}. 
}
\label{tab:summary-operations}
\end{table}


\begin{defn}[Network operations] \label{def:operations}
Consider two reaction networks ${G}=(\SS, \CC, \RR)$ and ${G'}=(\SS, \CC', \RR')$, both with species set $\SS=\{A_1, A_2, \dots, A_s\}$.  We say that $G'$ is obtained from $G$ by:
\begin{enumerate} 
%
\item {\em relabeling species} if there exists a permutation $\sigma$ of $\{1,2, \dots, s\}$
 such that the complexes and reactions of $G'$ are obtained from those of $G$ via $\sigma$; that is, 
 $ y_{1} A_1 + y_{2} A_2 + \cdots + y_{s} A_s \in \CC$ 
 if and only if 
 $y_{ \sigma(1)} A_{\sigma(1)} + 
 y_{ \sigma(2)} A_{\sigma(2)} + \cdots + 
 y_{ \sigma(s)} A_{\sigma(s)} \in \CC'$,
 and similarly for $\RR'$.
%
\item {\em translation by} 
$z = (z_1,z_2,..., z_s) \in \mathbb{Z}^s$, if the complexes and reactions of $G'$ are obtained by translating those of $G$ by 
$\widetilde z := z_1 A_1 + z_2 A_2 + \dots + z_s A_s$, 
that is, 
$\CC' = \{ y + \widetilde z \mid y \in \CC\}$ and 
$\RR' = \{ (y+ \widetilde z, y'+ \widetilde z) \mid (y,y') \in \RR \} $.

\item {\em stretching a reaction} if $\RR'$ is obtained from $\RR$ by replacing some reaction $(y, y') \in \RR$ by 
a reaction $(y, \widetilde{y}')$ such that the reaction vectors $y'-y$ and $\widetilde{y}'-y$ are positive-scalar multiples of each other.

\item {\em duplicating a reaction} if $\RR'$ is obtained from $\RR$ by replacing one reaction $(y, y') \in \RR$ by two reactions $(y, w)$ and $(y, z)$ that are not in $\RR$ 
such that $w-y$ and $z-y$ are both positive-scalar multiples of $y'-y$.  

%
\item {\em partial scaling} 
if, for some species $A_i$ and some $\alpha \in \mathbb{R} \smallsetminus \{0\}$, each reaction $(y , y')$ in~$\RR$ is replaced by $(y, \widetilde y' )$, where 
$\widetilde y' = \widetilde y'_1 A_1 +  \widetilde y'_2 A_2 + \dots + \widetilde y'_s A_s $ is defined by $\widetilde y'_j:=y'_j$ if $j \neq i$ and $\widetilde y'_j:=y_i +\alpha (y'_i-y_i)$ 
(equivalently, 
$\widetilde y'_j := 
\alpha y'_i + (1-\alpha) y_i$)
if $j=i$.
\end{enumerate}
\end{defn}


\begin{remark} \label{rem:operations-prior-lit}
Many of the operations in Definition~\ref{def:operations} have appeared in prior works.  For instance, translation and the fact that it preserves the positive steady-state locus appeared in works of Dickenstein~\cite{invitation}
and Boros, Craciun, and Yu~\cite{boros-craciun-yu}.  Also, the ``scalar multiplication'' operation in~\cite{boros-craciun-yu} is closely related to partial scaling, and the stretching operation was mentioned in~\cite[\S 4]{mv-small-networks}.  
Indeed, most of the properties of operations detailed in this section are well known, but for completeness we present the relevant results. 
\end{remark}

\begin{remark} \label{rem:johnston}
Our network operation of translation is a special case of the ``network translations'' introduced by Johnston~\cite{translated, tonello2017network}.  We translate the entire network, while Johnston allows translating parts of the network (typically linkage classes).  
\end{remark}

\begin{remark}[Relation to identifiability of networks] 
The operations of stretching and duplicating reactions affect the mass-action ODEs
at the level of the rate constants, such that the new system is ``dynamically equivalent''
to the original (see Lemma~\ref{lem:stretching}).  
Hence, as is well known, the mass-action ODEs (for a given choice of positive rate constants) are not enough to identify the specific network.  
Indeed, the fact that the operations of stretching and duplicating yield networks that are ``confoundable'' with the original network follows easily from results of Craciun and Pantea~\cite{CraciunPantea}.
\end{remark}

\begin{remark}[Adding networks] \label{rem:add-networks}
Boros, Craciun, and Yu considered the operation of ``adding'' networks, that is, taking unions to obtain the species, complexes, and reactions~\cite{boros-craciun-yu}.  This operation need not preserve ACR.  For instance, by adding the networks $\{0 \lra A\}$ and $\{ 2A \lra 3A\}$, both of which have ACR, we obtain the network 
$\{0 \lra A,~ 2A \lra 3A\}$, which does not (see Proposition~\ref{prop:1-species}).  
\end{remark}

\begin{remark} \label{rem:full-scaling}
Scaling the ODEs of {\em all} species by a single nonzero scalar $\alpha$ 
is often seen in the literature. This procedure  
can be viewed in the network either as applying partial scaling to each species, one at a time, 
or (if $\alpha>0$) as stretching all reactions by a factor of $\alpha$ 
(cf.~\cite[\S 3.2]{boros-craciun-yu}).
\end{remark}

\begin{remark} \label{rem:negate}
For a given network $G$, applying certain operations may be impossible; 
for instance, stretching a reaction might force the reaction arrow out of nonnegative orthant $\mathbb{R}^s_{\geq 0}$.  In such cases, it may be useful to first translate $G$ farther into the orthant before applying other operations.
\end{remark}

\begin{remark}[Reverse operations] \label{rem:reverse-operation}
It is straightforward to see that, each operation in Definition~\ref{def:operations}, except for duplicating a reaction, can be reversed by another version of the same operation. 
Also, these operations are easily seen to commute, as long as at each step a network is obtained (recall Remark~\ref{rem:negate}). 
\end{remark}

The main result of this section is as follows.

\begin{theorem}[Operations and ACR] \label{thm:operations-ACR}
    Consider two reaction networks $G$ and $G'$, such that $G'$ is obtained from $G$ by one of the operations in Definition~\ref{def:operations}.  
    Then $G$ has ACR if and only $G'$ has ACR.
\end{theorem}

We prove Theorem~\ref{thm:operations-ACR} in Section~\ref{sec:operations-ACR}.

\subsection{Operations and reaction diagrams} \label{sec:operations-diagram}

The following result follows easily from Definition~\ref{def:operations}.

\begin{lemma}[Operations and reaction diagrams] \label{lem:translate}
The first four network operations
in Definition~\ref{def:operations}
correspond to (respectively) the following geometric operations on the reaction diagram: 
\begin{enumerate}
   \item reflecting the diagram across the hyperplane 
    $x_{i_1}=x_{j_1}$, then across 
    $x_{i_2}=x_{j_2}$, and so on, ending with 
    $x_{i_n}=x_{j_n}$, where $\sigma = (i_n,j_n)\dots (i_2,j_2)(i_1,j_1)$ is a decomposition of the permutation $\sigma$ into a product of transpositions.

    \item translating the diagram by the vector $z$. 

    \item replacing an arrow in the diagram with a shorter or longer arrow in the same direction and having the same source as the original arrow.
    
    \item replacing an arrow in the diagram with two arrows in the same direction and with the same source of the arrow.
    \end{enumerate}

\end{lemma}

The fifth network operation, partial scaling, is more complicated to interpret as an operation on reaction diagrams (each arrow is rotated by some angle and then rescaled).  Accordingly, we give this interpretation only for networks with one-dimensional stoichiometric subspaces (Proposition~\ref{prop:1-d-operation}).  The effect on the ODEs, however, is easy to state:

\begin{lemma}[Partial scaling and ODEs] \label{lem:partial-scaling}
Assume that $G$ and $G'$ are reaction networks such that 
$G'$ is obtained from $G$ by partial scaling of the species $A_i$ by $ \alpha \in \mathbb R \smallsetminus \{0\}$.  Then:
\begin{enumerate}
    \item 
    the mass-action ODEs~\eqref{eq:ODE-mass-action} of $G'$ are obtained from those of $G$ by replacing the ODE $\dot{x}_i = f_i$ by $\dot{x}_i  = \alpha f_i$ (and leaving all other ODEs unchanged), and 
    \item the stoichiometric subspace of $G$ and the stoichiometric subspace of $G'$ have the same dimension.
\end{enumerate}
\end{lemma}
\begin{proof}
Both statements follow from the fact that the stoichiometric matrix $\Gamma'$ of $G'$ is obtained from the stoichiometric matrix $\Gamma$ of $G$ by scaling the $i$-th row of $\Gamma$ by $\alpha$. 
We explain this more precisely as follows.  Recall that the mass-action ODEs~\eqref{eq:ODE-mass-action} for $G$ can be rewritten as  
$\dot{x}= \Gamma v(x)$ 
where $\Gamma $ is the stoichiometric matrix of $G$ (recall from Section~\ref{sec:CRN} that the columns of $\Gamma$ are the reaction vectors 
$y' - y$ for all reactions $(y,y')$ of $G$)
and $v(x)$ is the column vector of monomials arising from the reactant of each reaction.  The ODEs for $G'$ are similar:  $\dot{x}= \Gamma' v(x)$. 
For any 
species $A_j$, 
the $j$-th entry of the reaction vector arising from some reaction $(y, \widetilde{y}')$ of $G'$ is
$(\widetilde y' -y)_j = y_j'-y_j $ if $ j\neq i  $, and $(\widetilde y' -y)_i= y_i+\alpha(y_i'-y_i)-y_i =\alpha(y_i'-y_i)$. 
This means that $\Gamma '$ is obtained from $\Gamma$ by scaling the $i$-th row by $\alpha$.  
Part 1 of the lemma now follows easily. 
Next, part 2 follows from the fact that the dimension of the stoichiometric space of a network is the rank of its stoichiometric matrix, and multiplying a row of a matrix by a nonzero scalar preserves its rank. 
\end{proof}

\begin{remark}\label{rem:alpha-nonzero}
The proof of Lemma~\ref{lem:partial-scaling}
shows why (in Definition~\ref{def:operations}) we do not allow partial scaling by $\alpha=0$: doing so might decrease the dimension of the stoichiometric subspace.  Indeed, a one-species network would become an empty network (with no reactions).
\end{remark}

We illustrate Lemma~\ref{lem:partial-scaling} through the following example involving networks with one-dimensional stoichiometric subspaces.

\begin{example} \label{ex:canonical-operation}
The $n=1$ version of the generalized Shinar--Feinberg network is $\{ B\to A,~A+B \to 2B \}$. 
By applying the operation of partial scaling to species $B$ by a factor of $-1$ -- 
we obtain $\{ B \to A+2B, ~ A+B \to 0 \}$.  The two reaction diagrams are shown here:
\begin{center}
	\begin{tikzpicture}[scale=.75]
	\draw (-6,0) -- (-1.75, 0);
	\draw (-5,-.5) -- (-5, 2.75);
	\draw [->] (-5,1) -- (-4, 0);
	\draw [->] (-4,1) -- (-5, 2);
    \node [left] at (-5, 1) {$B$};
    \node [below] at (-4, 0) {$A$};
    \node [right] at (-4, 1) {$A+B$};
    \node [right] at (-5, 2) {$2B$};
	\draw (-1,0) -- (3.25, 0);
	\draw (0,-.5) -- (0, 2.75);
	\draw [->] (0,1) -- (1, 2);
	\draw [->] (1,1) -- (0, 0.07);
    \node [left] at (0, 1) {$B$};
    \node [above] at (1, 2) {$A+2B$};
    \node [right] at (1,1) {$A+B$};
    \node [below] at (0.2, 0) {$0$};
	\end{tikzpicture}
\end{center}
\end{example}

\begin{proposition}[Partial scaling and reaction diagrams for 1-dimensional networks] \label{prop:1-d-operation}
If $G$ is a one-dimensional network 
and 
$G'$ is a network obtained from $G$ by partial scaling (specifically, by scaling species $A_i$ by a factor of $\alpha >0$), 
then the reaction diagram of $G'$ is obtained from the reaction diagram of $G$ by:
\begin{enumerate}
    \item rotating every arrow $y\rightarrow y'$ by the same angle (but centered at $y$), and 
    \item stretching every arrow $y\rightarrow y'$ by the following factor (which depends on the arrow):
    \[\lambda_{y \to y'} ~=~ \sqrt{1+\frac{(\alpha ^2-1)(y_i'-y_i)^2}{\lVert y'-y\rVert ^2}}.\]
\end{enumerate}
\end{proposition} 


 \begin{proof}
 Take two reactions $y \to y'$ and $z \to z'$ in $G$.  
The stoichiometric subspace of $G$ is one-dimensional, so $z'-z = \beta (y' -y)$ for some $\beta \in \mathbb{R} \smallsetminus \{0\}$. This means that 
these two reaction arrows 
point in the same direction (more precisely, up to lengthening an arrow, one can be obtained by the other by translation) if $\beta >0$, and in opposite directions (that is, up to lengthening, one can be obtained from the other by translation and then reversing the direction of the arrow) if $\beta <0$. 

Next, let $y \to \widetilde{y}'$ and $z \to \widetilde{z}'$ denote the corresponding reactions in $G'$.
Let $\theta$ denote the angle (in the plane  $P$ containing the arrows $y \to y'$ and $y \to \widetilde{y}'$) by which $y \to y'$ is rotated to obtain $y \to \widetilde{y}'$.  To show that $z \to z'$ is also rotated by the angle $\theta$ (in the translation of the plane $P$ that contains $z \to z'$), 
we must prove that the arrows of $y \to \widetilde{y}'$ and $z \to \widetilde{z}'$
point in the same direction if $\beta >0$, and in opposite directions if $\beta <0$. 
Thus, 
it suffices to show that $\widetilde{z}'-z = \beta (\widetilde{y}' -y)$.  Indeed, this follows easily from the definition of partial scaling: 
 $(\widetilde{y}' - y)_j = (y'_j-y_j)$ if $j \neq i$ and 
$(\widetilde{y}' - y)_j = \alpha (y'_i-y_i)$ if $j = i$; and analogously for $z, z', \widetilde{z}'$.

Finally, the factor by which a reaction $y \to y'$ is scaled is straightforward to compute: 
    \[\lambda_{y \to y'} ~=~ \frac{\lVert \widetilde y' - y \rVert}{\lVert y' - y \rVert} 
    ~=~ \frac{\lVert  (y' - y)+( \alpha -1)(y'-y)_ie_i \rVert}{\lVert y' - y \rVert} 
    ~=~ \sqrt{1+\frac{(\alpha ^2-1)(y_i'-y_i)^2}{\lVert y'-y\rVert ^2}} ~ ,\]
    where $e_i$ is the vector with 1 in the $i$-th entry and 0 in all other entries. 
 \end{proof}

\begin{remark}
We do not know whether there is a suitable generalization of 
Proposition~\ref{prop:1-d-operation}
beyond one-dimensional networks.
\end{remark}

The next result, which describes what happens to the mass-action ODEs when a reaction is stretched or duplicated, follows easily from definitions.

\begin{lemma}[Reaction operations and ODEs] \label{lem:stretching}
Let $G$ and $G'$ be reaction networks.
\begin{enumerate}
    \item Assume that $G'$ is obtained from $G$ by stretching the reaction $(y_i, y_j)$ by a factor of $\alpha>0$. Then the mass-action ODEs~\eqref{eq:ODE-mass-action} of $G'$ are obtained from those of $G$ by replacing the rate constant $\kappa_{ij}$ by $\alpha \kappa_{ij}$.
    \item Assume that $G'$ is obtained from $G$ by duplicating a reaction 
    $(y_i, y_j)$ by way of reactions $(y_i, y_k)$ and $(y_i, y_{\ell})$ such that 
    $y_k - y_i = \alpha (y_j- y_i)$ and 
    $y_{\ell} - y_i = \beta (y_j- y_i)$, with $\alpha>0$ and $\beta>0$.  
    Then the mass-action ODEs~\eqref{eq:ODE-mass-action} of $G'$ are obtained from those of $G$ by
    replacing the rate constant $\kappa_{ij}$ by $\alpha \kappa_{ik} + \beta \kappa_{i \ell} $.
\end{enumerate}
\end{lemma}

We saw above that partial scaling preserves the dimension of the stoichiometric subspace (Lemma~\ref{lem:partial-scaling}).  It is straightforward to check that the other four network operations
in Definition~\ref{def:operations} also have this property, so we obtain the following result.

\begin{proposition}[Operations preserve dimension] \label{prop:4-opers-preserve-dim}
Let $G$ and $G'$ be reaction networks such that $G'$ is obtained from $G$ by performing one or more of the operations from Definition~\ref{def:operations}. 
Then the stoichiometric subspace of $G$ and the stoichiometric subspace of $G'$ have the same dimension.
\end{proposition}

\subsection{Operations and deficiency} \label{sec:operations-deficiency}

Certain network operations preserve the deficiency:

\begin{prop}[Operations and deficiency] \label{prop:operations-preserve-deficiency}
If $G'$ is obtained from a reaction network~$G$ by translation or relabeling species, then $G$ and $G'$ have the same deficiency. 
\end{prop}

\begin{proof}
Translation and relabeling species do not affect the number of complexes or the number of linkage classes.  The dimension of the stoichiometric subspace is also unchanged (Proposition~\ref{prop:4-opers-preserve-dim}).  The result now follows from the definition of deficiency.
\end{proof}

The remaining three operations from Definition~\ref{def:operations} do not, in general, preserve deficiency; see the following three examples.
Nevertheless, we will see that for the generalized Shinar-Feinberg network and certain other networks, 
the operations of partial scaling and stretching a reaction do preserve deficiency (Proposition~\ref{prop:canonical-deficiency}).

 \begin{example}[Stretching a reaction does not preserve deficiency] \label{ex:stretch-def}
Starting from the network $\{0 \to A,~ A \to B,~ B \to 0 \}$, 
we obtain the network
$\{0 \to 2A,~ A \to B, ~ B \to 0 \}$ by 
stretching the reaction $0 \to A$ by a factor of 2.  The first network has deficiency $3-1-2=0$, while the second network 
has deficiency $4-1-2=1$.
\end{example}
 
 \begin{example}[Duplicating a reaction does not preserve deficiency] \label{ex:duplicate-def} 
From the network $\{0 \to A\}$, 
we obtain $\{0 \to A, ~ 0 \to 2A\}$
by duplicating a reaction.  The first network has deficiency $2-1-1=0$, while the second network has deficiency $3-1-1=1$.
\end{example}

\begin{example}[Partial scaling does not preserve deficiency] \label{ex:partial-scale-def}
The networks $\{0 \lra 2A\}$ and $\{0 \to A \leftarrow 2A\}$ are obtained from each other by partial scaling of both reactions, but their deficiencies are not equal: the first has deficiency $2-1-1=0$ and the second has deficiency $3-1-1=1$. \end{example}

The following lemma reveals one case in which deficiency is preserved by network operations, involving networks with only two reactions.

\begin{lemma} \label{lem:deficiency-2-rxns}
Let $G$ be a reaction network with exactly two reactions, $y\to y'$ and $z \to z' $, such that
the reactant complexes are distinct ($y \neq z$) and there exists a species $A_i$ such that $y'_i \neq y_i = z_i \neq z_i'$.  If $G'$ is obtained from $G$ by performing one or more operations from Definition~\ref{def:operations}, except duplicating a reaction, then $G$ and $G'$ have the same deficiency.
\end{lemma}
\begin{proof}
The properties satisfied by $G$ are that  (1) there are exactly two reactions, (2)~the reactant complexes are distinct  ($y \neq z$), and (3) there exists a species $A_i$ such that $y'_i \neq y_i = z_i \neq z_i'$. 
It is straightforward to check that
the first two properties are preserved when one or more network operations, except duplicating a reaction, are performed. 
Now consider the third property. This property is preserved under relabeling species, translation, and stretching a reaction.
The last operation to consider is partial scaling.  
By definition, partial scaling some species $A_{\ell}$ (by a nonzero $\alpha$) replaces the reaction $y \to y'$ by 
 $y \to \widetilde y'$, where $\widetilde y'_j:=y'_j$ if $j \neq \ell$ and $\widetilde y'_j:=y_{\ell} +\alpha (y'_{\ell}-y_{\ell})$ if $j=\ell$.  The reaction $z \to z'$ is similarly transformed. 
 Hence, if $y'_i \neq y_i = z_i \neq z_i'$, then 
 $\widetilde y'_i \neq \widetilde y_i = \widetilde z_i \neq \widetilde z_i'$. In other words, 
 the third property still holds after partial scaling is performed.


We conclude that $G$ has 3 or 4 complexes, and $G'$ too has 3 or 4 complexes.  
A 2-reaction network $N$ with 3 complexes has only 1 linkage class, and so the deficiency of $N$ is $3-1-\dim(N)=2-\dim(N)$, where $\dim(N)$ denotes the dimension of the stoichiometric subspace of $N$.  
Similarly, a 2-reaction network $N$ with 4 complexes has 2 linkage classes, and so its deficiency is $4-2- \dim(N)= 2- \dim(N)$, which is the same as in the 3-complex case.  Now the result follows from the fact that $\dim(G')=\dim(G)$ (by Lemma~\ref{prop:4-opers-preserve-dim}).
\end{proof}

Lemma~\ref{lem:deficiency-2-rxns} implies the following result, which we will use in a later section.

\begin{proposition}[Operations preserve deficiency of Shinar--Feinberg network] \label{prop:canonical-deficiency}
Let $G'$ be a reaction network obtained from a generalized Shinar--Feinberg network
$G=\{ B \to A, ~ nA+B \to (n-1)A+2B \}$ by performing one or more operations from Definition \ref{def:operations}, except duplicating a reaction.  Then $G$ and $G'$ have the same deficiency.
\end{proposition}

\subsection{Operations and ACR} \label{sec:operations-ACR}

Some network operations preserve the steady-state locus: 

\begin{prop}[Operations and the steady-state locus] \label{prop:operations-preserve}
Assume that $G'$ is obtained from a reaction network $G$ by one of the operations in  Definition~\ref{def:operations}.  Let $\kappa^*$ be a vector of positive rate constants for $G$.
\begin{enumerate}
    \item If $G'$ is obtained from $G$ by translation or partial scaling, then the mass-action systems $(G, \kappa^*)$ and $(G', \kappa^*)$ have the same steady-state locus.
    \item If $G'$ is obtained from $G$ by relabeling species, 
    then the steady-state locus of $(G', \kappa^*)$ is obtained from that of $(G, \kappa^*)$ by performing reflections across some hyperplanes of the form $x_i=x_j$.
\end{enumerate}

\end{prop}

\begin{proof} Translation simply multiplies every equation on the right-hand side of the ODEs~\eqref{eq:ODE-mass-action} by a Laurent monomial, which 
does not affect the positive steady states. 
Next, partial scaling does not change the positive steady states, as the nonzero constant can be eliminated when solving for steady states (recall Lemma~\ref{lem:partial-scaling}).  Relabeling species (permuting variables) only permutes the axes when considering the steady-state locus, thus the steady states are preserved up to a renaming of variables. 
\end{proof}

We can now prove that all the network operations in Definition~\ref{def:operations} preserve ACR.
\begin{proof}[Proof of Theorem~\ref{thm:operations-ACR}]  
In this proof, we denote a reaction by $(y_i, y_j)$ rather than $(y,y')$.

For translation, partial scaling, and relabeling species, this result follows from Proposition~\ref{prop:operations-preserve}.  
Now we consider the remaining operations.  Accordingly, assume that $G'$ is obtained from $G$ by  stretching or duplicating a reaction.  

First, we consider the case of stretching: assume that a reaction $(y_{i^*}, y_{j^*})$ is stretched by a factor of $\alpha > 0$.  Recall that this stretching operation can be reversed by another stretching operation (Remark~\ref{rem:reverse-operation}), so it suffices to show that if $G$ has ACR in species $A_{p}$, then so does $G'$.  By a straightforward application of Lemma~\ref{lem:stretching}, we know that 
$x^* \in \mathbb{R}^s_{>0}$ is a positive steady state of 
$(G, \kappa^*)$ if and only if 
$x^* \in \mathbb{R}^s_{>0}$ is a positive steady state of 
$(G', \widetilde{\kappa}^*)$, where $\widetilde{\kappa}^*_{ij}$ is 
defined to be $\frac{1}{\alpha} \kappa^*_{i^* j^*}$
if $(i,j)=(i^*,j^*)$, and equals
$\kappa^*_{i j}$
otherwise.
Now it follows from the definition of ACR that 
$G'$ also has ACR in species $A_{p}$ if $G$ does.

Next, we consider the case of duplicating a reaction.  By Lemma~\ref{lem:stretching}, the ODEs of $G'$ are obtained from those of $G$ by replacing some $\kappa_{i^* j^*}$ by some 
$\alpha \kappa_{i^* k} + \beta \kappa_{i^* \ell}$ (with $\alpha>0$ and $\beta>0$).  
Thus, 
$x^* \in \mathbb{R}^s_{>0}$ is a positive steady state of 
$(G, \kappa^*)$ if and only if 
$x^* \in \mathbb{R}^s_{>0}$ is a positive steady state of 
$(G', \widetilde{\kappa}^*)$ for every choice of $\widetilde{\kappa}^*$ for which 
(i)
${\kappa}^*_{i^* j^*}   = \alpha \widetilde{\kappa}_{i^* k} + \beta \widetilde{\kappa}_{i^* \ell} $, 
and (ii) if $(i,j) \notin \{ (i^*,k), (i^*, \ell) \}$, then $\widetilde{\kappa}^*_{ij} = \kappa^*_{i j}$
It is now straightforward to conclude that $G$ has ACR in some species if and only if $G'$ does too.
\end{proof}


\section{One-species networks} \label{sec:1-species}
A network with only one species has ACR if and only if it is non-multistationary; such networks have already been classified~\cite{JS}.  
We can therefore characterize one-species networks with ACR (Proposition~\ref{prop:1-species}) and stable ACR (Proposition~\ref{prop:1-species-stable}).

To state Proposition~\ref{prop:1-species}, we need the following definition from~\cite{JS}:

\begin{definition} \label{def:2-sign-change}
A one-species network is {\em 2-alternating} if it has exactly three reactions and the arrow diagram is either $(\to, \leftarrow, \to)$ or $(\leftarrow, \to, \leftarrow )$.
\end{definition}

\begin{proposition}[ACR in one-species networks] \label{prop:1-species}
 For a one-species network $G$ with at least one reaction, the following are equivalent:
 \begin{enumerate}
 \item $G$ has ACR,
 \item $G$ is non-multistationary, and 
 \item $G$ does not have a 2-alternating subnetwork, and 
	 the arrow diagram of $G$ does not have the form 
	 $(\lradot, \dots, \lradot)$. 
\end{enumerate}
\end{proposition}

\begin{proof}
This result follows directly from the definitions of ACR and multistationary, and from~\cite[Theorem 3.6, part 2]{JS}, 
which implies that a one-species network $G$ (with at least one reaction) is multistationary if and only if 
$G$ has a 2-alternating subnetwork or 
	 the arrow diagram of $G$ has the form $(\lradot, \dots, \lradot)$. 
\end{proof}

Recall that networks that do not admit positive steady states (e.g., $0 \to A$) trivially have ACR (Remark~\ref{rem:no-pos-steady-states}).  Such trivial networks are ruled out in the next result (we assume that a positive steady state exists), which pertains to stable ACR.  

\begin{proposition}[Stable ACR] \label{prop:1-species-stable}
  For a one-species network $G$ with at least one reaction, the following are equivalent:  
 \begin{enumerate}
 \item $G$ has stable ACR and admits a positive steady state, and
 \item the arrow diagram of $G$ has one of the following four forms:
    \begin{align} \label{eq:4-arrow-diagram-forms}
         (\lradot,~ \la, \la, \dots, \la)~, 
         &
        \quad 
        (\to, \to, \dots, \to, ~\lradot) ~, 
        \\
        \notag
        (\to, \to, \dots, \to, ~ \lradot, ~ \la, \la, \dots, \la)~,
        & \quad \quad 
        {\textrm or} \quad
        (\to, \to, \dots, \to, ~ \la, \la, \dots, \la)~.
    \end{align}
%
 \end{enumerate}
\end{proposition}
\begin{proof}
We begin by proving $2 \Rightarrow 1$.  Assume that the arrow diagram of $G$ has one of the forms shown in~\eqref{eq:4-arrow-diagram-forms}.
In particular, the arrow diagram does not have the form $(\lradot, \dots, \lradot)$, and it also follows that $G$ has no 2-alternating subnetwork.  Thus, by Proposition~\ref{prop:1-species}, $G$ has ACR.  

Next, we prove that $G$ admits a positive steady state.  
We may assume that $G$ has only one species: $|\SS|=1$ (all species not appearing in the reactions may be disregarded without affecting the existence of positive and/or stable steady states).  
The arrow diagram of $G$ guarantees that there exists a vector of positive rate constants $\kappa^*=(\kappa_{ij})$ such that the right-hand side of the ODE $\frac{dx_A}{dt} = f_{\kappa^*}(x_A)$ has the following form:
\begin{align} \label{eq:ODE-1-species}
    f_{\kappa^*}(x_A) ~=~ \mu_1 x_A^{p_1} + \mu_2 x_A^{p_2} + \dots +  \mu_m x_A^{p_m}~,
\end{align}
where $m \geq 2$, $\mu_1 >0$ and $\mu_m<0$, and also $0 \leq p_1 < p_2 < \dots < p_m$.  Thus, 
    $f_{\kappa^*}(x_A) > 0$
for $x_A>0$ sufficiently small, 
while 
    $f_{\kappa^*}(x_A) < 0$
for $x_A$ sufficiently large.  Hence, the intermediate-value theorem implies that $f_{\kappa^*}(x^*_A) = 0$ for some $x^*_A>0$, that is, a positive steady state $x_A^*$ of $(G,\kappa^*)$ must exist.

What remains to show is that every positive steady state is exponentially stable.  Accordingly, assume that $x_A^* \in \mathbb{R}_{>0}$ is a steady state of $(G,\kappa^*)$ for some $\kappa^* \in \mathbb{R}^{|\mathcal{R}]}$. 
Then, as $G$ has ACR, $x_A^*$ is the unique positive steady state of $(G,\kappa^*)$.  So, by Descartes' rule of signs and the form of the arrow diagrams~\eqref{eq:4-arrow-diagram-forms}, the right-hand side of the ODE, as in~\eqref{eq:ODE-1-species}, has exactly 1 sign change and $x_A^*$ is a multiplicity-one root of $f_{\kappa^*}(x_A)$.  The same argument as above therefore shows that     $f_{\kappa^*}(x_A) $ is positive for $x_A>0$ sufficiently small and negative for $x_A$ large.  
We conclude that 
$f_{\k^*}$ is decreasing at $x^*_A$. 
Hence,  
$f'_{\kappa^*}(x^*_A)<0$ 
(here, $f'_{\kappa^*}(x^*_A) = 0$ is precluded because $x_A^*$ has multiplicity one), that is, $x_A^*$ is exponentially stable, as desired.

We prove $1 \Rightarrow 2$.  Assume that $G$ has ACR, and the arrow diagram does \uline{not} have one of the forms
shown in~\eqref{eq:4-arrow-diagram-forms}.  So, by Proposition~\ref{prop:1-species} and the fact that $G$ has at least one reaction, the arrow diagram must have one of the following forms:
\begin{enumerate}[(i)]
        \item  $(\to, \to, \dots, \to)$, 
        \item  $(\la, \la, \dots, \la)$, 
        \item $(\la, \la, \dots, \la,~ \to, \to, \dots, \to)$,
        \item $(\lradot,~ \to, \to, \dots, \to)$, 
        \item  $(\la, \la, \dots, \la, ~\lradot)$, or
        \item  $(\la, \la, \dots, \la, ~\lradot,~ \to, \to, \dots, \to)$.
\end{enumerate}
We must show that every resulting ODE system, if it has a positive steady state, is unstable.  
When the arrow diagram has the form in 
(i) or (ii), there are no positive steady states.  Next, assume that the arrow diagram has the form in (iii), (iv), (v), or (vi).  In this case, by applying Descartes' rule of signs to the the right-hand side of the ODE as in~\eqref{eq:ODE-1-species} (like what was done earlier in the proof), we see that each resulting ODE system, if there is a positive steady state $x_A^*$, then $x_A^*$ is unstable.  This completes the proof.
\end{proof}

\begin{example} \label{ex:1-species}
The network $G=\{0 \lra A~,~ 2A\to 3A\}$ has arrow diagram $(\to, \la, \to)$ and hence is itself a 2-alternating network.  So, 
by Proposition~\ref{prop:1-species}, $G$ does not have ACR.
The network  $H=\{0 \lra A \to 2A\}$ has arrow diagram $(\to, \lradot)$. 
Thus, by  
Proposition~\ref{prop:1-species-stable},   $H$ has stable ACR (in species $A$).
\end{example}

\begin{example}[Example~\ref{ex:1-species-4-reactions}, continued] \label{ex:1-sp-4-rxn-cont}
The one-species network from Example~\ref{ex:1-species-4-reactions} has the following 2-alternating subnetwork: 
$\{ A_1 \to 2A_1~,~ A_1 \leftarrow 2A_1~,~ 3A_1 \to 4A_1\}$.  Thus, Proposition~\ref{prop:1-species} implies what we saw earlier: the network does not have ACR.
\end{example}

Next, we analyze the special case of one-species networks with only two reactions.
\begin{proposition}[ACR in 1-species, 2-reaction network] \label{prop:1-species-2-reactions}
  Let $G$ be a one-species network with exactly two reactions.  
Then the following are equivalent:
        \begin{enumerate}
            \item $G$ has ACR,
            \item the arrow diagram of $G$ is either $(\to, \leftarrow)$ or $(\leftarrow, \to)$, and
            \item $G$ can be obtained by network operations (Definition~\ref{def:operations})
            from the network $\{0 \lra nA \}$, for some $n \geq 1$.
        \end{enumerate}
\end{proposition}

\begin{proof}
The equivalence $1 \Leftrightarrow 2$ follows easily from Proposition~\ref{prop:1-species}. 

Next, we prove the implication $1 \Rightarrow 3$.  We consider two cases.  First, assume that the arrow diagram of $G$ is $(\to, \leftarrow)$. 
So, after relabeling species if necessary, $G=\{aA \to bA~,~ cA \leftarrow dA \}$ for some non-negative integers $a,b,c,d$ with $a<b$, $a<d$, and $c<d$.  Thus, $n:=d-a \geq 1$.  Recall that all operations, except duplicating a reaction, are reversible (Remark~\ref{rem:reverse-operation}).  So, it suffices to transform $G$ by operations to the network $\{0 \lra nA\}$.  Accordingly, we first stretch each reaction of $G$ to have length $n=d-a$, and so obtain $\{aA \lra dA \}$.  Next, we translate the network to the left by $a=d-n$, which yields $\{0 \lra nA\}$, and so we are done.

Now consider the remaining case, when $G$ has arrow diagram
 $(\leftarrow, \to)$.  Hence, $G=\{ aA \leftarrow bA~,~ cA \to dA \}$, for some non-negative integers $a,b,c,d$ with $a<b<c<d$.  If 
 $2c<d$,
 translate the network to the right by $d-2c$ to obtain 
 $\{a'A \leftarrow b'A~,~ c'A \to d'A \}$ with $2c' \geq d'$ (in fact, equality holds).  
 If $2c \geq d$, 
 simply relabel $a,b,c,d$ by (respectively) $a',b',c',d'$.  
 Next, apply the operation of partial scaling by $-1$ to the species $A$, which yields $\{b'A \to (b'-a')A ~,~ (2c'-d')A \leftarrow c'A \}$, which satisfies (by construction) $b' < b'-a'$, $b'<c'$, and $0 \leq 2c'-d' = c'- (d'-c') < c'$.  Hence, this final network has arrow diagram 
$(\to, \leftarrow)$, and so we have reduced to the prior case, which we already solved above.

Finally, the implication $3 \Rightarrow 1$ holds: 
we already noted that the networks $\{ 0 \lra nA\}$ have ACR (Example~\ref{ex:acr-but-not-shinar--Feinberg}), and network operations preserve ACR (Theorem~\ref{thm:operations-ACR}).
\end{proof}

Finally, we consider translations of one-species networks with two reactions. 
The following result follows easily from Proposition~\ref{prop:1-species-2-reactions}.

\begin{corollary} \label{cor:translate-1-species-2-reactions}
Let  $G$ be a reaction network that is obtained by translation from a one-species network that has exactly two reactions.  Then the following are equivalent:
      \begin{enumerate}
            \item $G$ has ACR,
            \item $G$ can be obtained by network operations (Definition~\ref{def:operations})
            from the network $\{0 \lra nA \}$, for some $n \geq 1$.      \end{enumerate}
\end{corollary}

\section{One-dimensional networks} \label{sec:main-results}
In the prior section, Proposition~\ref{prop:1-species-stable} 
characterized stable ACR for one-species networks. Our goal here is to generalize that result to one-dimensional networks (see Theorem~\ref{thm:1-dim-stable}). 

We begin with two lemmas, which imply that ACR in a one-dimensional network constrains the arrow diagrams of its one-species embedded networks.  Recall that a reactant complex $y$ is such that $y \to y'$ is a reaction for some complex $y'$.

\begin{lemma} \label{lem:1-dim-reactants}
Let $G$ be a one-dimensional network with species $A_1, A_2, \dots, A_s$.
Assume that $G$ admits a positive steady state.  
If $G$ has ACR in some species $A_{i^*}$, then 
the reactant complexes of $G$ differ only in species $A_{i^*}$ 
(that is, if $y$ and $\widetilde y$ are both reactant complexes of $G$, then $y_i=\widetilde{y}_i$ for all $i \neq i^*$).
\end{lemma} 
\begin{proof}
Let $G$ be a one-dimensional network that admits a positive steady state and has ACR in species $A_{i^*}$.  By relabeling, if necessary, we may assume that $i^*=1$.  Assume for contradiction that there exist two reactant complexes that differ in some species $A_i$ with $i \geq 2$.  
By relabeling species, if needed, we may assume that $i=2$.  
It follows that the embedded diagram obtained from $G$ by removing all species except $A_2$, which we denote by $N_2$, has at least two distinct reactant complexes.  
%
%

By hypothesis, $G$ admits a positive steady state, so there exist 
      positive rate constants 
            $k_1^*, k_2^*, \dots, k_r^*$ 
      for which $G$ admits a positive steady state $\mathbf{x}^* = (x_1^*, x_2^*, \dots, x_s^*)$.   
Let $g$ denote the right-hand side of the mass-action ODE for $x_2$.  By assumption, $x_2^*$ is a positive root of the following univariate (in $x_2$) polynomial, obtained by specializing $g$:
\begin{align} \label{eq:specialize-g}
    \widetilde g ~:=~ g|_{k_1=k_1^*, k_2=k_2^*, \dots, k_r =k_r^*,~ x_1=x_1^*, x_3=x_3^*, x_4=x_4^*, \dots, x_s=x_s^*  }~.
\end{align}
Recall that $N_2$ has at least two reactant complexes.  So, unless the arrow diagram of $N_2$ has the form $(\lradot, \dots, \lradot)$, the polynomial $\widetilde g$ is not the zero polynomial.  If, on the other hand, the arrow diagram is $(\lradot, \dots, \lradot)$, then make another choice for the  $k_i^*$ and $x^*_i$, as follows.  The form of the arrow diagram and the fact that there are at least two reactant complexes ensure that we can pick $k_1^*, k_2^*, \dots, k_r^*>0$ and set  $x_1^* =x_3^* =x_4^* = \dots =x_s^* =1$ such that the new version of $\widetilde{g}$, as in~\eqref{eq:specialize-g}, has at least two terms and is such that the leading coefficient is positive and the lowest-order coefficient is negative.  Hence, by Descartes' rule of signs, there exists a positive root $x^*_2$ of $\widetilde g$.

Let $x_2^{\alpha}$ denote the lowest power of $x_2$ that appears in the polynomial $\widetilde{g}  $.  It follows that $h:= \frac{\widetilde{g}}{x_2^{\alpha}}$ is a univariate polynomial that has $x^*_2$ as a root and has a nonzero constant term. 

We make the following claim:

\noindent
{\bf Claim:}
We may assume that $x_2^*$ is a multiplicity-one root of $h$.  

We prove the claim as follows.  
If $x_2^*$ is a multiple root of $h$ (that is, $h' (x^*_2)=0$), then we simply need to translate the graph of $h$ up or down a small amount in order for a new positive root (near $x_2^*$), with multiplicity one, to appear.  (Here we are using the fact that $h$ is a non-constant polynomial.)  We achieve this translation by replacing one of the rate constants $k_i^*$ that contributes to the constant term of $h$ (notice that each rate constant contributes to only one monomial) by $k_i^*+ \epsilon$ or  $k_i^*- \epsilon$ (for sufficiently small $\epsilon$).  We also ensure that $\epsilon$ is sufficiently small so that  $k_i^*- \epsilon >0 $.  Our claim therefore holds.

Now, our claim implies that by perturbing $x_1^*$, we preserve a (positive) root (near $x_2^*$).  More precisely, there exists a small $\delta>0$, such that the following univariate polynomial has a positive root, which we denote by $x_2^{**}$:
\begin{align*}
    \frac { g|_{k_1=k_1^*, k_2=k_2^*, \dots, k_r =k_r^*,~ 
x_1=x_1^* + \delta, x_3=x_3^*, x_4=x_4^*, \dots, x_s=x_s^*  }
}
{x_2^{\alpha}}~.
\end{align*}
We conclude that $(x_1^* + \delta, x_2^{**}, x_3^*, x_4^*, \dots, x_s^*)  $ is a positive steady state of $(G, \kappa^*)$, where 
$\kappa^*:=(k_1^*, k_2^*, \dots, k_r^*)$.   However, 
$(x_1^*, x_2^{*}, \dots, x_s^*)  $ is also positive steady state of $(G, \kappa^*)$, and so $G$ does not have ACR in $A_1$.  This is a contradiction. %
\end{proof}

\begin{remark} \label{rem:1-reactant-embed-net} 
The condition from Lemma~\ref{lem:1-dim-reactants} that the reactant complexes of $G$ differ only in the species $A_{i^*}$
implies that the embedded network of $G$ obtained by removing $A_{i^*}$
has at most one reactant complex (the embedded network has no reactant complexes if it has no reactions).
\end{remark}

\begin{lemma}[ACR for one-dimensional networks] \label{lem:1-dim}
 Let $G$ be a one-dimensional network with species $A_1, A_2, \dots, A_s$.
Assume that $G$ admits a positive steady state.  
If $G$ has ACR in some species $A_{i^*}$ that is \uline{not} a catalyst-only species, then the following holds:
	\begin{enumerate}[(a)] 	
	\item  the embedded network of $G$ obtained by removing all species except $A_{i^*}$, which we denote by $N$, 
	does not have a 2-alternating subnetwork, and 
	 the arrow diagram of $N$ does not have the form 
	 $(\lradot, \dots, \lradot)$, and
	\item the reactant complexes of $G$ differ only in species $A_{i^*}$ 
(that is, if $y$ and $\widetilde y$ are both reactant complexes of $G$, then $y_i=\widetilde{y}_i$ for all $i \neq i^*$).
	\end{enumerate}
Moreover, if $G$ has stable ACR, then 
	 the arrow diagram of $N$ has one of these four forms: 
    \begin{align} \label{eq:4-arrow-diagram-forms-1-d}
         (\lradot,~ \la, \la, \dots, \la)~, 
         &
        \quad 
        (\to, \to, \dots, \to, ~\lradot) ~, 
        \\
        \notag
        (\to, \to, \dots, \to, ~ \lradot, ~ \la, \la, \dots, \la)~,
        & \quad \quad 
        {\textrm or} \quad
        (\to, \to, \dots, \to, ~ \la, \la, \dots, \la)~.
    \end{align}
\end{lemma}

\begin{proof}
Let $G$ be a one-dimensional network that admits a positive steady state and has ACR in some non-catalyst-only species $A_{i^*}$.  By relabeling species, if needed, we may assume $i^*=1$.  Let $N$ denote the embedded network obtained by removing all species except $A_1$.
Part (b) of the lemma follows from Lemma~\ref{lem:1-dim-reactants}, so below we verify part (a).

 By hypothesis, there exist 
      positive rate constants 
            $k_1^*, k_2^*, \dots, k_r^*$ 
      for which $G$ admits a positive steady state $\mathbf{x}^* = (x_1^*, x_2^*, \dots, x_s^*)$.  
    We specialize the ODE for $x_1$, as follows:
    \begin{align} \label{eq:ODE-specialized} 
    \notag
        \frac{dx_1}{dt}|_{x_2=x_2^*, \dots, x_s=x_s^*}
        ~&=~
        k_1 (y'_{11} - y_{11}) x_1^{y_{11}} (x_2^*)^{y_{12}} \dots  (x_s^*)^{y_{1s}}
        + \dots +
        k_r (y'_{r1} - y_{r1}) x_1^{y_{r1}} (x_2^*)^{y_{r2}} \dots  (x_s^*)^{y_{rs}}  \\
        &=~ 
       \widetilde{k_1} (y'_{11} - y_{11}) x_1^{y_{11}} 
       + \dots +
       \widetilde{k_r} (y'_{r1} - y_{r1}) x_1^{y_{r1}} ~,
    \end{align}
where we have defined new parameters via 
$\widetilde{k_i} =k_i (x_2^*)^{y_{i2}} \dots  (x_s^*)^{y_{is}} $, for $i=1,2,\dots, r$ (the product $(x_2^*)^{y_{i2}} \dots  (x_s^*)^{y_{is}} $ is positive because $x^*$ is a positive steady state).
The expression in~\eqref{eq:ODE-specialized} 
is exactly the right-hand side of the ODE for the embedded network $N$
(except that a summand arising from a reaction $z \to z'$ in $N$ might appear more than once if more than one reaction of $G$ projects $z \to z'$).
We observe that $N$ admits a positive steady state, 
because $x_1^*$ is a positive root of the right-hand side of~\eqref{eq:ODE-specialized} when 
the values of the new parameters are 
$\widetilde{k^*_i} := {k^*_i} (x_2^*)^{y_{i2}} \dots  (x_s^*)^{y_{is}} $ (for $i=1,2,\dots, r$).  
It is also straightforward to show that, because $G$ has 
ACR (respectively, stable ACR) in $A_1$, then $N$ does too 
(the fact that $A_1$ is not a catalyst-only species ensures that~\eqref{eq:ODE-specialized} is not identically zero). 
Also, $N$ has at least one reaction (because $A_1$ is not a catalyst-only species).
Thus, 
by Proposition~\ref{prop:1-species} (respectively, Proposition~\ref{prop:1-species-stable}), 
the network $N$ has the properties described in part (a) (respectively, the arrow diagram of $N$ has one of the four forms in~\eqref{eq:4-arrow-diagram-forms-1-d}).
\end{proof}

We will show that a partial converse to Lemma~\ref{lem:1-dim} holds
(Theorem~\ref{thm:1-dim-stable}), but first we need a lemma that pertains to catalyst-only species.  
We saw earlier that the degenerate-ACR network has ACR in a catalyst-only species, but all steady states are degenerate (Example~\ref{ex:daniele-3}).  This phenomenon holds more generally, as the following lemma attests.

\begin{lemma}[ACR in a catalyst-only species] \label{lem:cat-only-and-ACR}
If a one-dimensional reaction network $G$ has ACR in some catalyst-only species $A_{i^*}$, then, for all choices of positive rate constants, every positive steady state is degenerate and so $G$ does \uline{not} have stable ACR in $A_{i^*}$.
\end{lemma}

\begin{proof}
Let $s$ and $r$ denote the numbers of species and reactions, respectively.  By relabeling, if necessary, we may assume that $i^*=1$. 

Fix a choice of positive rate constants $\kappa \in \mathbb{R}^r_{>0}$.  Let ${\mathbf{x}^*}= (x^*_1, x^*_2, \dots, x^*_s) \in \mathbb{R}_{>0}^s$ be a positive steady state of the mass-action system $(G,\kappa)$.  We must show that ${\mathbf{x}^*}$ is a degenerate steady state. 
As the stoichiometric subspace $S$ is one-dimensional, it suffices to show the following equality 
(where $ \mathbf{0}$ denotes the zero vector in $\mathbb{R}^s$):
\begin{align} \label{eq:degenerate}
    {\rm Im}\left( df_{\kappa} (\mathbf{x}^*)|_{S} \right) ~=~ \{ \mathbf{0} \} ~.
\end{align}

By hypothesis, $A_1$ is a catalyst-only species.  So, by Lemma~\ref{lem:cat-only-species}, the mass-action ODE for $A_1$ is $\frac{dx_1}{dt} = 0$.  Hence, for all $z \in S$, we have $z_1 =0$.  So, to prove the equality~\eqref{eq:degenerate}, we need only show that columns $2,3,\dots, s$ of the matrix $df_{\kappa} (\mathbf{x}^*)$ are zero columns.

We first examine row-1 of $df_{\kappa} (\mathbf{x}^*)$.  The corresponding ODE, as we saw, is $\frac{dx_1}{dt} = 0$.  So, row-1 of the matrix is the zero row.

Our next task is to analyze one of the remaining rows, that is, row-$j$ of $df_{\kappa} (\mathbf{x}^*)$, where $2 \leq j \leq s$.  As $G$ is one-dimensional and has ACR in $A_1$, Lemma~\ref{lem:1-dim-reactants} implies that the reactant complexes of $G$ differ only in species $A_1$.  Hence, the mass-action ODE for species $A_j$ is either 0 (in which case row-$j$ is, as desired, the zero row) or can be written as:
\begin{align} \label{eq:ode-for-species-j}
    \frac{dx_j}{dt} ~=~ x_2^{n_2} x_3^{n_3} \dots x_s ^ {n_s} ~ g(x_1)~,
\end{align}
for some $n_2,n_3,\dots, n_s \in \mathbb{Z}_{\geq 0}$ and $g \in \mathbb{R}[x_1]$.  By construction,  ${x}_1^*$ is a positive root of the right-hand side of~\eqref{eq:ode-for-species-j} and so is a root of $g$.

Equation~\eqref{eq:ode-for-species-j} implies that in 
row-$j$ of the Jacobian matrix $df_{\kappa} (x)$ (before evaluating at $\mathbf{x}^*$), the $k$-th entry, for $2 \leq k \leq s$, is either 0 or a polynomial multiple of $g(x_1)$. Hence, after evaluating at $\mathbf{x}^*$, these entries all become $0$. 
We conclude that, as desired, columns $2,3,\dots,s$ of $df_{\kappa} (\mathbf{x}^*)$ are zero.
\end{proof}

\begin{theorem}[Stable ACR for one-dimensional networks] \label{thm:1-dim-stable}
 Let $G$ be a network with species $A_1, A_2, \dots, A_s$.
 If the stoichiometric subspace of $G$ is one-dimensional, then the following are equivalent: 
 \begin{enumerate}
 \item $G$ has stable ACR and admits a positive steady state, and
 \item there exists a species $A_{i^*}$ such that the following holds:
	\begin{enumerate} 	
	\item for the embedded network of $G$ obtained by removing all species except $A_{i^*}$, 
	 the arrow diagram has one of these four forms:
    \begin{align}  \label{eq:4-diagrams-again}
         (\lradot,~ \la, \la, \dots, \la)~, 
         &
        \quad 
        (\to, \to, \dots, \to, ~\lradot) ~, 
        \\
        \notag
        (\to, \to, \dots, \to, ~ \lradot, ~ \la, \la, \dots, \la)~,
        & \quad \quad 
        {\textrm or} \quad
        (\to, \to, \dots, \to, ~ \la, \la, \dots, \la)~, 
    \end{align}
	\item the reactant complexes of $G$ differ only in species $A_{i^*}$ 
(that is, if $y$ and $\widetilde y$ are both reactant complexes of $G$, then $y_i=\widetilde{y}_i$ for all $i \neq i^*$).
	\end{enumerate}
 \end{enumerate}
\end{theorem}

\begin{proof}
The implication $1 \Rightarrow 2$ follows from Lemmas~\ref{lem:1-dim} and~\ref{lem:cat-only-and-ACR}.
Now we prove $2 \Rightarrow 1$. Denote the reactions of $G$ by $y_i \to y_i'$, with rate constant $k_i$, for $i=1,2,\dots, r$.  By relabeling, if necessary, we may assume that $i^*=1$.  By assumption, 
    the stoichiometric coefficient for $A_2$ (respectively, for $A_3, \dots, A_s$)
    is the same in all reactant complexes.  We denote these coefficients by
    $z_2, z_3, \dots, z_s$, respectively. 
    Thus, we can rewrite each monomial $x^{y_i}$, for $i=1,2,\dots, r$, as $x^{y_i}=x_1^{y_{i1}} (x_2^{z_2} x_3^{z_3} \dots x_s^{z_s} )$.  The right-hand side of the ODE for $x_1$ can now be written as follows:
    \begin{align} \label{eq:ode-x-1}
        (x_2^{z_2} x_3^{z_3} \dots x_s^{z_s} ) ~
            \left(
            k_1 (y'_{11} -y_{11}) x_1^{y_{11}}
            +
            k_2 (y'_{21} -y_{21}) x_1^{y_{21}}
            +
            \dots 
            +
            k_r (y'_{r1} -y_{r1}) x_1^{y_{r1}}
            \right)~.
    \end{align}
    The expression~\eqref{eq:ode-x-1} is not identically zero, because of the form~\eqref{eq:4-diagrams-again} of the embedded network, which we denote by $N$, obtained by removing all species except $A_1$.  Hence, as $G$ is one-dimensional, the right-hand sides of the remaining $(s-1)$ ODEs are scalar multiples of~\eqref{eq:ode-x-1}.  We conclude that the positive steady states of $G$ correspond to the positive roots of the following polynomial in $x_1$ (the monomial factor in~\eqref{eq:ode-x-1} can be disregarded): 
    \begin{align} \label{eq:poly-1-vble}
                    k_1 (y'_{11} -y_{11}) x_1^{y_{11}}
            +
            k_2 (y'_{21} -y_{21}) x_1^{y_{21}}
            +
            \dots +
            k_r (y'_{r1} -y_{r1}) x_1^{y_{r1}}~.
    \end{align}
    This polynomial~\eqref{eq:poly-1-vble} is exactly
    the right-hand side of the ODE of the (one-species) embedded network $N$ (except that a summand might appear more than once if more than one reaction of $G$ projects to some reaction in $N$), 
    and $N$ has arrow diagram one of those in~\eqref{eq:4-diagrams-again}.  So, by Proposition~\ref{prop:1-species-stable}, $N$ has stable ACR (in $A_1$) and admits a positive steady state.  It is straightforward to show now that the same is true for $G$.  
    \end{proof}
\begin{example}[Example~\ref{ex:gen-canonical-network}, continued] \label{ex:gen-canonical-network-2}
We saw earlier that the generalized Shinar--Feinberg network is one-dimensional and admits a positive steady state, and the embedded network obtained by $A$ (respectively, $B$) has arrow diagram $(\lradot)$ (respectively, $(\to, \leftarrow)$).  
Thus, Theorem~\ref{thm:1-dim-stable} implies that $G$ has stable ACR. 
\end{example}

We saw earlier that one-reaction networks do not admit positive steady states (Remark~\ref{rem:no-pos-steady-states}). On the other hand, two-reaction networks that admit a positive steady state must be one-dimensional (Lemma~\ref{lem:2-rxn-steady-state-implies-1-d}), and so the above results apply.  We pursue this topic in the next section.

\section{Networks with two reactions} \label{sec:2-species-2-rxns}
This section focuses on networks with two reactions.  
Some of these networks, like $B \lra 2A + B$, 
contain catalyst-only species.
Our first result characterizes ACR for such networks, when there are only two species (Proposition~\ref{prop:2-rxn-cat-only}). 
Next, we focus on the remaining 2-species, 2-reaction networks, 
and 
show that
up to network operations, the only networks with ACR 
are the generalized Shinar--Feinberg networks from Example~\ref{ex:generalized-shinar--Feinberg} (Theorem~\ref{thm:main-result-2-rxn-2-species}).  
It follows that, up to network operations, there are only three families of two-species, two-reaction networks (possibly containing catalyst-only species) with ACR (Corollary~\ref{cor:2-rxn-2-species}).  
More generally, when there are two reactions (and any number of species, but no catalyst-only species) ACR is completely characterized by the Shinar--Feinberg criterion (Theorem~\ref{thm:2-rxn}).  

\subsection{Networks with catalyst-only species}

\begin{proposition}[ACR for networks with 2 species, 2 reactions, and a catalyst-only species] \label{prop:2-rxn-cat-only}
Let  $G$ be a reaction network 
with exactly two species and two reactions, such that one of the species is a catalyst-only species.  
Assume that $G$ admits a positive steady state.
Then the following are equivalent:
      \begin{enumerate}
            \item $G$ has ACR,
            \item $G$ can be obtained by network operations (Definition~\ref{def:operations})
            from the network $\{0 \lra mA \}$, for some $m \geq 1$, 
            or from the degenerate-ACR network 
    $\{ A \to 2A, ~ A+ nB \to nB \}$, for some $n \geq 1$.            
        \end{enumerate}
\end{proposition}

\begin{proof}
We have already seen that $2 \Rightarrow 1$ holds; see Corollary~\ref{cor:translate-1-species-2-reactions},   Example~\ref{ex:daniele-3}, and Theorem~\ref{thm:operations-ACR}. 

Now we prove $1 \Rightarrow 2$.  Let $y\to y'$ and $z \to z'$ denote the two reactions of $G$.  
By relabeling species, if necessary, we may assume that $B$ is a catalyst-only species (that is, $y_B=y'_B$ and $z_B=z'_B$).  Let $N$ be the embedded network obtained by removing species $B$.  
If $N$ has no reactions, then $A$ too is a catalyst-only species and so $G$ has no reactions, which is a contradiction.  So, $N$ has either one or two reactant complexes.  We consider these two cases separately.

Assume that $N$ has only one reactant complex.  So, the arrow diagram of $N$ is $(\to)$, $(\la)$, or $(\lradot)$.   The first two arrow diagrams are impossible (it is straightforward to see that either would imply that $G$ does not admit positive steady states).  So, the arrow diagram must be $(\lradot)$.  This means that $y_A=z_A$ and $y'_A - y_A$ and 
$z'_A - z_A$ are both nonzero and have opposite signs.

We consider two subcases, based on whether $y_B = z_B$ or $y_B \neq z_B$.  First assume that $y_B=z_B$.  This equality, together with 
$y_A=z_A$, imply that 
the mass-action ODE for $A$ can be written as follows:
\begin{align} \notag
    \frac{dx_A}{dt} ~&=~ \kappa_1 (y_A' - y_A) x_A^{y_A} x_B^{y_B} + \kappa_2   (z_A' - z_A) x_A^{z_A} x_B^{z_B} \\
                    ~&=~ \left( \kappa_1 (y_A' - y_A)  + \kappa_2   (z_A' - z_A) \right) x_A^{y_A} x_B^{y_B}~.
                    \label{eq:ode-for-A-subcase}
\end{align}
Now choose $\kappa_1^*:=1$ and $\kappa_2^*:=-(y_A'-y_A)/(z_A'-z_A)$ (which is positive, as we observed earlier that $y'_A - y_A$ and 
$z'_A - z_A$ have opposite signs).  It follows that, for these rate constants $\kappa_1^*, \kappa_2^*$, the right-hand side of~\eqref{eq:ode-for-A-subcase} is 0, and so every $(x_A,x_B) \in \mathbb{R}^s_{>0}$ is a positive steady state.  Hence, $G$ does not have ACR, which is a contradiction.

The remaining subcase is when $y_B \neq z_B$.  We may assume that $y_B < z_B$.  Let $n:= z_B - y_B$.  We will perform operations on $G$ and obtain a degenerate-ACR network.  We begin by stretching each reaction so the reaction vector has length $1$.  Next, if $y_A'< y_A$, then perform a partial scaling of species $A$ by a factor of $-1$.  The resulting reactions are $(y_A, y_B) \to (y_A+1, y_B)$ and $(y_A, y_B+n) \to (y_A - 1, y_B+n)$.  Now translate left by $y_A-1$ and down by $y_B$ to obtain the degenerate-ACR network $\{ A \to 2A, ~ A+ nB \to nB \}$.  

Now consider the remaining case, when $N$ has exactly two reactant complexes (that is, $x_A \neq y_A$). 
Lemma~\ref{lem:1-dim-reactants} implies that $x_B = y_B$.  This equality, together with the fact that $B$ is a catalyst-only species, implies that $G$ is obtained by translation from a 1-species network.  So, by Corollary~\ref{cor:translate-1-species-2-reactions}, we can perform network operations to obtain
$\{0 \lra mA \}$, for some $m \geq 1$, as desired.
\end{proof}

\subsection{Networks with no catalyst-only species}

\begin{theorem}[ACR for networks with 2 species, 2 reactions, and no catalyst-only species]
    \label{thm:main-result-2-rxn-2-species}
    Let $G$ be a reaction network with exactly two species ($A_1$ and $A_2$) and two reactions. Assume that $G$ 
    admits a positive steady state
    and 
    has no catalyst-only species. 
    Let $N_1$ (respectively, $N_2$) denote the embedded network of $G$ obtained by removing species $A_1$ (respectively, $A_2$). 
    Then the following are equivalent:
    \begin{enumerate}
        \item $G$ has ACR;
        \item one of $N_1$ and $N_2$ has arrow diagram $(\lradot)$, 
        and the other has arrow diagram either $(\to, \leftarrow)$ or $(\leftarrow, \to)$;
        %
    \item $G$ can be obtained by network operations (Definition~\ref{def:operations}) from 
    a generalized Shinar--Feinberg network $\{ B \to A+2B, ~ nA+B \to (n-1)A \}$, for some $n \geq 1$; and
        \item $G$ satisfies the Shinar--Feinberg criterion (Proposition~\ref{prop:shinar--Feinberg-criterion-network}).
    \end{enumerate}
Moreover, stable ACR corresponds to the ``$(\to , \leftarrow)$'' case from part 2.
\end{theorem}

\begin{proof}
    Assume $G$ is a 2-reaction, 2-species network that admits a positive steady state and has no catalyst-only species.   Lemma~\ref{lem:2-rxn-steady-state-implies-1-d} implies that $G$ is a one-dimensional network. 

The implication $4 \Rightarrow 1$ is Proposition~\ref{prop:shinar--Feinberg-criterion-network}, so the rest of the proof proceeds by proving the implications $1 \Rightarrow 2 \Rightarrow 3 \Rightarrow 4$.

 We first prove $1 \Rightarrow 2$.  
        Assume that $G$ has ACR.  Then, by Lemma~\ref{lem:1-dim}, 
        one of the embedded networks $N_1$ and $N_2$ has arrow diagram 
         $(\to, \leftarrow)$ or $(\leftarrow, \to)$.
         Lemma~\ref{lem:1-dim} also implies that the other embedded network must have exactly one reactant complex (there must be at least one reaction, as $G$ has no catalyst-only species) and so has arrow diagram $(\lradot)$, $(\to)$, or $(\leftarrow)$.  However, the cases of $(\to)$ and $(\leftarrow)$ are ruled out, because otherwise (it is straightforward to see) $G$ would not admit a positive steady state.
        
  Next, we prove $2 \Rightarrow 3$.
        Assume $G$ satisfies condition 2 of the theorem, that is, 
        either $N_1$ or $N_2$ has arrow diagram $(\lradot)$, 
        and the other has arrow diagram $(\to, \leftarrow)$ or $(\leftarrow, \to)$.
        Recall that all operations, except duplicating a reaction, are reversible (Remark~\ref{rem:reverse-operation}).  So, it suffices to transform $G$ by operations to a generalized Shinar--Feinberg network.  To accomplish this, our first step is to relabel species, if necessary, to obtain a network in which $N_1$ has arrow diagram 
        $(\lradot)$ and $N_2$ has arrow diagram $(\to, \leftarrow)$ or $(\leftarrow, \to)$.
        Next, if the lengths of the reaction vectors are unequal, then rescale the longer reaction so that it has the same length as the shorter reaction. Notice that the arrow diagrams of $N_1$ and $N_2$ are unaffected.
        
        The fact that $N_2$ has arrow diagram $(\to, \leftarrow)$ or $(\leftarrow, \to)$ implies that we can write the reactions as $y\to y'$ and $z \to z'$, with $y_1<z_1$.  Let $n:=z_1-y_1$.  (We also have $y_2=z_2 \geq 1$, because $N_1$ has arrow diagram $\lradot$.)
        
        We claim that both components of the reaction vector of $y \to y'$, that is, $(y'-y)_1$ and $(y'-y)_2$, are nonzero.  Indeed, otherwise $N_1$ or $N_2$ would have at most one reaction, and this would contradict our hypothesis.  Thus, we can now perform
        partial scaling to species $A_1$ by a factor of $\frac{1}{(y'-y)_1}$ and then partial scaling to species $A_2$ by $\frac{-1}{(y'-y)_2}$. 
        
        By construction, our reactions, which we denote by $y\to \widetilde{y}'$ and $z \to \widetilde{z}'$, satisfy $\widetilde{y}' - y=(1,-1)$ and (as the reactions have the same length but point in opposite directions) $\widetilde{z}' - z= (-1,1)$.  We conclude that our reactions are as follows:
        \begin{align*}
            y_1 A_1 + y_2 A_2 ~ &\to (y_1+1)A_1 + (y_2-1) A_2\\
            (y_1+n) A_1 + y_2 A_2 ~ &\to (y_1+n-1)A_1 + (y_2+1) A_2~.
        \end{align*}
        An earlier observation, $y_2 \geq 1$, implies that the above reactions do not leave the non-negative orthant. 
        Finally, translating left by $y_1$ and down by $(y_2-1)$ yields $\{A_2 \to A_1~,~ nA_1+A_2 \to (n-1)A_1 + 2 A_2\}$, which is a generalized Shinar--Feinberg network.

We now prove $3 \Rightarrow 4$.   
        We saw in Example~\ref{ex:gen-canonical-network}
that every mass-action system arising from one of the 
generalized Shinar--Feinberg networks
$\{ B \to A+2B, ~ nA+B \to (n-1)A \}$, for $n \geq 1$, 
satisfies the 
Shinar--Feinberg criterion.  So, we need only show that after performing a sequence of network operations  (excluding duplicating, as we are considering only networks with two reactions) to a generalized Shinar--Feinberg network, (i) the deficiency is unchanged, and (ii) the two reactant complexes differ only in one of the two species.  First, (i) holds by 
Proposition~\ref{prop:canonical-deficiency}.  Also, (ii) is easy to verify: relabeling species simply swaps the role of the species, translation maintains the difference between the two reactant complexes, and stretching and partial scaling preserve the reactant complexes.

         Finally, we consider stability.  By Lemma~\ref{lem:1-dim}, 
         if $G$ has stable ACR, then we are in the case of ``$(\to , \leftarrow)$''.  And, conversely, it is straightforward to check that the ``$(\leftarrow, \to)$'' case yields a positive steady state that is unstable and hence precludes stable ACR.
\end{proof}

Theorem~\ref{thm:main-result-2-rxn-2-species} yields a geometric criterion for ACR, as follows.  Recall that an embedded network can be viewed as a projection of the geometric diagram onto a coordinate axis (Remark~\ref{rem:projection}).  Accordingly, 
condition 2 in Theorem~\ref{thm:main-result-2-rxn-2-species} 
requires that, up to relabeling species, the geometric diagram 
has one of the forms shown in Figure~\ref{fig:ACR2spe2rxn}.  

\begin{figure}[ht]
    \centering
    \includegraphics[scale=0.42]{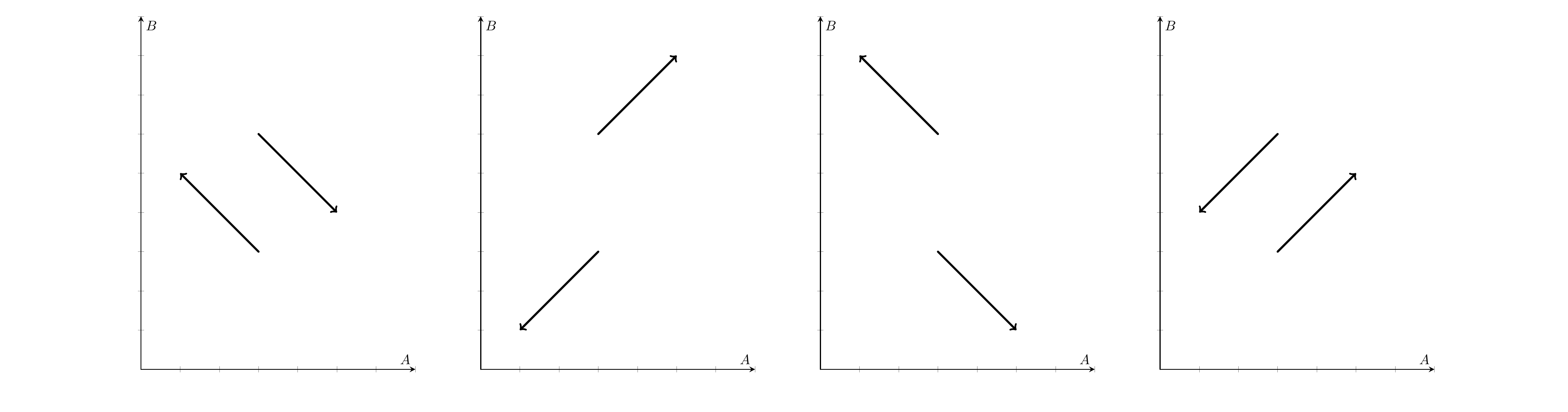}
    \caption{Possible geometric diagrams of a network with two species, two reactions, and no catalyst-only species, and having ACR in species $B$. The projection to $A$ yields the arrow diagram $( \protect \lradot )$, and the projection to $B$ yields $( \leftarrow, \to )$ or $( \to, \leftarrow)$.}
    \label{fig:ACR2spe2rxn}
\end{figure}

\begin{example} [Example~\ref{ex:1}, continued]
We revisit the network 
$\{ 3A + 5B \to A + 6B~, ~A + 3B \to 3A + B \}$.   
The reaction diagram (Figure~\ref{fig:geom_rep}) does not (even after relabeling species) have one of the forms shown in 
Figure~\ref{fig:ACR2spe2rxn}.  So, the network does not have ACR. 
Alternatively, we can draw the same conclusion using Theorem~\ref{thm:main-result-2-rxn-2-species}. 
Indeed, the embedded network obtained by removing $A$ (respectively, $B$) has arrow diagram $(\leftarrow, \to)$ (respectively, $(\to, \leftarrow)$) -- neither arrow diagram is $(\lradot)$ -- so Theorem~\ref{thm:main-result-2-rxn-2-species} implies that there is no ACR.  
\end{example}

\begin{example} \label{ex:procedure}
    Consider the network 
        $\{5A + B \to 7A,~ 5A + 3B \to A + 5B\}$, with reaction diagram as follows:
        
    \begin{center}
    	\begin{tikzpicture}[scale=.75]
    	\draw (-1,0) -- (8, 0);
    	\draw (0,-.5) -- (0, 5.25);
    	\draw [->] (5,3) -- (1, 5);
    	\draw [->] (5,1) -- (7, 0.07);
        \node [right] at (5, 3) {$5A + 3B$};
        \node [right] at (1.2, 5) {$A+5B$};
        \node [left] at (5,1) {$5A+B$};
        \node [above] at (7, 0.2) {$7A $};
    	\end{tikzpicture}
    \end{center}
    Here
     we 
    follow the steps in the proof of Theorem~\ref{thm:main-result-2-rxn-2-species}, to perform operations to transform this network into a generalized Shinar--Feinberg network.  We begin by relabeling the species (switching the roles of $A$ and $B$), which yields 
     $\{A + 5B \to 7B,~ 3A + 5B \to 5A + B\}$.  Next, we rescale the second reaction by $1/2$, so both reactions in the resulting network have the same length, as shown in the following reaction diagram (the axes are not shown to avoid excess white space):
     \begin{center}
    	\begin{tikzpicture}[scale=.75]
    	\draw [->] (1,5) -- (0.07, 7);
    	\draw [->] (3,5) -- (4,3);
        \node [below] at (1, 5) {$A + 5B$};
        \node [right] at (0, 7) {$7B$};
        \node [right] at (3,5) {$3A+5B$};
        \node [right] at (4, 3) {$4A +3B $};
    	\end{tikzpicture}
    \end{center}
    Next, we apply partial scaling to species $A$ by a factor of $1/(-1)=-1$ and then partial scaling to species $B$ by $-1/2$.  This yields:
     \begin{center}
    	\begin{tikzpicture}[scale=.75]
    	\draw [->] (1,5) -- (2, 4);
    	\draw [->] (3,5) -- (2,6);
        \node [left] at (1, 5) {$A + 5B$};
        \node [right] at (2, 4) {$2A+4B$};
        \node [right] at (3,5) {$3A+5B$};
        \node [above] at (2, 6) {$2A +6B $};
    	\end{tikzpicture}
    \end{center}
    Finally, we translate left by $1$ and down by $4$ to obtain $\{B\to A,~ 2A+B \to A+2B\}$, which is the $n=2$ generalized Shinar--Feinberg network.
\end{example}   

\subsection{General results}

The following result follows directly from Proposition~\ref{prop:2-rxn-cat-only} and Theorem~\ref{thm:main-result-2-rxn-2-species}. 

\begin{corollary}[ACR for networks with 2 species and 2 reactions]
    \label{cor:2-rxn-2-species}
    Let $G$ be a reaction network with exactly two species and two reactions. Assume that $G$ 
    admits a positive steady state.
Then $G$ has ACR if and only if 
	 $G$ can be obtained by network operations (Definition~\ref{def:operations}) from one of the following networks:
\begin{itemize}
	\item a network $\{0 \lra mA \}$, for some $m \geq 1$, 
         \item a degenerate-ACR network 
$\{ A \to 2A, ~ A+ nB \to nB \}$, for some $n \geq 1$, and
	\item a generalized Shinar--Feinberg network $\{ B \to A+2B, ~ pA+B \to (p-1)A \}$, for some 
	 $p \geq 1$.
\end{itemize}	
\end{corollary}

Our final result characterizes ACR for networks with two reactions (and at least two species).  It generalizes part of Theorem~\ref{thm:main-result-2-rxn-2-species} (in fact, all of that theorem can be generalized but the statements are somewhat technical, so we omit them).

\begin{theorem}[ACR for networks with 2 reactions]
    \label{thm:2-rxn}
    Let $G$ be a reaction network with exactly two reactions and at least two species. Assume that $G$ 
    admits a positive steady state
    and has no catalyst-only species. 
Then the following are equivalent:
    \begin{enumerate}
        \item $G$ has ACR;
\item $G$ satisfies the Shinar--Feinberg criterion (Proposition~\ref{prop:shinar--Feinberg-criterion-network}).
    \end{enumerate}
\end{theorem}

\begin{proof}
The implication ``$2 \Rightarrow 1$'' is Proposition~\ref{prop:shinar--Feinberg-criterion-network}.  

For ``$1 \Rightarrow 2$'', assume that $G$ is as in the statement of the theorem and has ACR.
As $G$ admits a positive steady state, Lemma~\ref{lem:2-rxn-steady-state-implies-1-d} implies that $G$ is one-dimensional.  
For $i=1,2,\dots,s$, where $s$ is the number of species, let $N_i$ denote the one-species embedded network obtained from $G$ by removing all species except $i$.  By Lemma~\ref{lem:1-dim}, one of these networks $N_{i^*}$ has arrow diagram $(\to, \leftarrow)$ or $(\leftarrow, \to)$, and every remaining $N_j$ (for $j\neq i^*$) has 
at most one reactant complex.  In fact, each of these $N_j$ must have exactly one reactant complex (otherwise, some species $A_j$ would be a catalyst-only species) and so has arrow diagram $(\lradot)$, $(\to)$, or $(\leftarrow)$.  However, the cases of $(\to)$ and $(\leftarrow)$ are ruled out, because otherwise $G$ would not admit a positive steady state.  Now the arrow diagram $(\lradot)$ implies that $G$ has either 3 or 4 complexes.  
If $G$ has 3 complexes, there is only one linkage class and so the deficiency is $\delta = 3-1-1=1$.  If $G$ has 4 complexes, there are two linkage classes and so we get the same deficiency: $\delta=4-2-1=1$.  Thus, to show that $G$ satisfies the Shinar--Feinberg criterion, we need only show that the two reactant complexes (which are by definition non-terminal because each is the source of a reaction) differ in species $i^*$ and not in any species $j$ (for $j \neq i^*$).  Indeed, the reactants differ in $i^*$ because $N_{i^*}$ has arrow diagram $(\to, \leftarrow)$ or $(\leftarrow, \to)$; and the reactants do not differ in 
$j$, for $j \neq i^*$, because $N_j$ has 
only one reactant complex.
\end{proof}

%

\section{Discussion} \label{sec:discussion}
Detecting ACR is, in general, difficult.  Nevertheless, here we showed that this task is straightforward for certain classes of networks.  For one-species and (more generally) one-dimensional networks, ACR can be assessed easily from arrow diagrams or reaction diagrams (Theorem~\ref{thm:1-dim-stable}).  
For networks with only two reactions 
and no catalyst-only species, ACR is characterized by the Shinar--Feinberg criterion (Theorem~\ref{thm:2-rxn}). And, for networks with exactly two reactions, two species, and no catalyst-only species, we discovered that, up to network operations, the only networks with ACR are the generalized Shinar-Feinberg networks (Theorem~\ref{thm:main-result-2-rxn-2-species}).

As noted in the introduction, our interest in ACR for small networks is due to the potential of these networks to represent robust modules within larger biological systems~\cite{controller,shinar2010structural}.  
Thus, it is desirable to classify small networks with ACR.  Our results contribute to this important direction.  

Going forward we need more theoretical and computational results on detecting ACR.  Are there more examples of networks in which ACR can be easily assessed?  Our results suggest that answers to this question may use 
the Shinar--Feinberg criterion or its generalizations~\cite{karp2012complex, tonello2017network}, and may involve
in a crucial way geometric representations of the network.  Indeed, we are very much interested in more results in which ACR can be ``seen'' from the reaction diagram.

\subsection*{Acknowledgements}
{\footnotesize
Part of this project was initiated by NM and AS at a SQuaRE (Structured Quartet Research Ensemble) at AIM. 
Also, some of this research was completed while AT was a visiting scholar at Texas A\&M University.  
NM and AS thank Luis Garc\'ia Puente, Elizabeth Gross, Heather Harrington, and Matthew Johnston for helpful discussions. Also, the authors are grateful to Daniele Cappelletti, Elisenda Feliu, 
Beatriz Pascual Escudero,
and Elisa Tonello for sharing ideas and examples. 
NM was partially supported by the NSF (DMS-1853525) and a Clare Boothe Luce professorship. 
AS was partially supported by the 
NSF (DMS-1752672) and the 
Simons Foundation (\#521874).
AT was partially supported by the Independent Research Fund Denmark and the Wallenberg AI, Autonomous Systems and Software Program (WASP) funded by the Knut and Alice Wallenberg Foundation.  
The authors acknowledge two referees for their helpful comments.}

\bibliographystyle{plain}
\bibliography{crnt}

\begin{thebibliography}{10}

\bibitem{anderson2019discrepancies}
David~F Anderson and Daniele Cappelletti.
\newblock Discrepancies between extinction events and boundary equilibria in
  reaction networks.
\newblock {\em J.\ Math.\ Biol.}, 79(4):1253--1277, 2019.

\bibitem{Anderson2016}
David~F. Anderson, Daniele Cappelletti, and Thomas~G. Kurtz.
\newblock Finite time distributions of stochastically modeled chemical systems
  with absolute concentration robustness.
\newblock {\em SIAM J. Appl. Dyn. Syst.}, 16(3):1309--1339, 2017.

\bibitem{A-E-J}
David~F. Anderson, German Enciso, and Matthew~D. Johnston.
\newblock Stochastic analysis of chemical reaction networks with absolute
  concentration robustness.
\newblock {\em J. R. Soc. Interface}, 11(93):20130943, 2014.

\bibitem{boros-craciun-yu}
Bal{\'a}zs Boros, Gheorghe Craciun, and Polly~Y. Yu.
\newblock Weakly reversible mass-action systems with infinitely many positive
  steady states.
\newblock {\em Preprint, {\tt arXiv:1912.10302}}, 2019.

\bibitem{hidden}
Daniele Cappelletti, Ankit Gupta, and Mustafa Khammash.
\newblock A hidden integral structure endows absolute concentration robust
  systems with resilience to dynamical concentration disturbances.
\newblock {\em J.\ R.\ Soc.\ Interface}, 17(171):20200437, 2020.

\bibitem{CraciunSIAGA}
Gheorghe Craciun.
\newblock Polynomial dynamical systems, reaction networks, and toric
  differential inclusions.
\newblock {\em SIAM Journal on Applied Algebra and Geometry}, 3(1):87--106,
  2019.

\bibitem{CraciunPantea}
Gheorghe Craciun and Casian Pantea.
\newblock Identifiability of chemical reaction networks.
\newblock {\em J.\ Math.\ Chem.}, 44(1):244--259, 2008.

\bibitem{dexter2015}
Joseph~P. Dexter, Tathagata Dasgupta, and Jeremy Gunawardena.
\newblock Invariants reveal multiple forms of robustness in bifunctional enzyme
  systems.
\newblock {\em Integr.\ Biol.}, 7:883--894, 2015.

\bibitem{dexter2013dimerization}
Joseph~P Dexter and Jeremy Gunawardena.
\newblock Dimerization and bifunctionality confer robustness to the isocitrate
  dehydrogenase regulatory system in {Escherichia coli}.
\newblock {\em J.\ Biol.\ Chem.}, 288(8):5770--5778, 2013.

\bibitem{invitation}
Alicia Dickenstein.
\newblock Biochemical reaction networks: {A}n invitation for algebraic
  geometers.
\newblock In {\em Mathematical Congress of the Americas}, volume 656, pages
  65--83. American Mathematical Soc., 2016.

\bibitem{eloundou2016network}
Jeanne~MO Eloundou-Mbebi, Anika K{\"u}ken, Nooshin Omranian, Sabrina Kleessen,
  Jost Neigenfind, Georg Basler, and Zoran Nikoloski.
\newblock A network property necessary for concentration robustness.
\newblock {\em Nat.\ Commun.}, 7:13255, 2016.

\bibitem{Enciso2016}
German~A. Enciso.
\newblock Transient absolute robustness in stochastic biochemical networks.
\newblock {\em J. R. Soc. Interface}, 13(121):20160475, 2016.
\newblock doi: 10.1098/rsif.2016.0475.

\bibitem{feinberg-def0}
Martin Feinberg.
\newblock Chemical reaction network structure and the stability of complex
  isothermal reactors {I}. {T}he deficiency zero and deficiency one theorems.
\newblock {\em Chem.\ Eng.\ Sci.}, 42(10):2229--68, 1987.

\bibitem{ProjArg}
Manoj Gopalkrishnan, Ezra Miller, and Anne Shiu.
\newblock A projection argument for differential inclusions, with applications
  to persistence of mass-action kinetics.
\newblock {\em SIGMA Symmetry Integrability Geom. Methods Appl.}, 9:Paper 025,
  25, 2013.

\bibitem{translated}
Matthew Johnston.
\newblock Translated chemical reaction networks.
\newblock {\em Bull.\ Math.\ Biol.}, 76(5):1081--1116, 2014.

\bibitem{joshi-shiu-I}
Badal Joshi and Anne Shiu.
\newblock Simplifying the {J}acobian criterion for precluding multistationarity
  in chemical reaction networks.
\newblock {\em SIAM J. Appl. Math.}, 72:857--876, 2012.

\bibitem{JS}
Badal Joshi and Anne Shiu.
\newblock Which small reaction networks are multistationary?
\newblock {\em SIAM J.\ Appl.\ Dyn.\ Syst.}, 16(2):802--833, 2017.

\bibitem{karp2012complex}
Robert~L Karp, Mercedes~P{\'e}rez Mill{\'a}n, Tathagata Dasgupta, Alicia
  Dickenstein, and Jeremy Gunawardena.
\newblock Complex-linear invariants of biochemical networks.
\newblock {\em J.\ Theoret.\ Biol.}, 311:130--138, 2012.

\bibitem{controller}
Jinsu Kim and German Enciso.
\newblock Absolutely robust controllers for chemical reaction networks.
\newblock {\em J. R. S. Interface}, 17(166):20200031, 2020.

\bibitem{neigenfind2011biochemical}
Jost Neigenfind, Sergio Grimbs, and Zoran Nikoloski.
\newblock Biochemical network decomposition reveals absolute concentration
  robustness.
\newblock {\em Preprint, {\tt arXiv:1105.0624}}, 2011.

\bibitem{neigenfind2013relation}
Jost Neigenfind, Sergio Grimbs, and Zoran Nikoloski.
\newblock On the relation between reactions and complexes of (bio) chemical
  reaction networks.
\newblock {\em J.\ Theoret.\ Biol.}, 317:359--365, 2013.

\bibitem{mv-small-networks}
Nida Obatake, Anne Shiu, and Dilruba Sofia.
\newblock Mixed volume of small reaction networks.
\newblock {\em Involve, a Journal of Mathematics}, 13:845--860, 2020.

\bibitem{beatriz-elisenda}
Beatriz Pascual-Escudero and Elisenda Feliu.
\newblock Local and global robustness in systems of polynomial equations.
\newblock {\em Preprint, {\tt arXiv:2005.08796}}, 2020.

\bibitem{Mercedes}
Mercedes~Soledad P\'erez~Mill\'an.
\newblock {\em M\'etodos algebraicos para el estudio de redes bioqu\'imicas}.
\newblock PhD thesis, Universidad de Buenos Aires, 2011.

\bibitem{shinar2010structural}
Guy Shinar and Martin Feinberg.
\newblock Structural sources of robustness in biochemical reaction networks.
\newblock {\em Science}, 327(5971):1389--1391, 2010.

\bibitem{tang-xu}
Xiaoxian Tang and Hao Xu.
\newblock Multistability of small reaction networks.
\newblock {\em Preprint, {\tt arXiv:2008.03846}}, 2020.

\bibitem{tonello2017network}
Elisa Tonello and Matthew~D. Johnston.
\newblock Network translation and steady state properties of chemical reaction
  systems.
\newblock {\em B.\ Math.\ Biol.}, 80(9):2306--2337, 2018.

\bibitem{wiuf2020classification}
Carsten Wiuf and Chuang Xu.
\newblock Classification and threshold dynamics of stochastic reaction
  networks.
\newblock {\em Preprint, {\tt arXiv:2012.07954}}, 2020.

\bibitem{xu2020criteria}
Chuang Xu, Mads~Christian Hansen, and Carsten Wiuf.
\newblock Criteria for dynamics of one-dimensional continuous time {M}arkov
  chains with applications.
\newblock {\em Preprint, {\tt arXiv:2006.10548}}, 2020.

\end{thebibliography}

\end{document}